\documentclass[11pt]{elsarticle}
\usepackage{amssymb,amsfonts,amsmath,mathrsfs,mathtools}
\usepackage{graphicx,epsfig,subfigure}
\usepackage{bm}
\usepackage[margin=1in,papersize={8.5in,11in}]{geometry}
\usepackage{times}
\usepackage{multirow}
\usepackage{color}

\def\vs{\vspace{0.2cm}}

\newcommand\blfootnote[1]{%
  \begingroup
  \renewcommand\thefootnote{}\footnote{#1}%
  \addtocounter{footnote}{-1}%
  \endgroup
}

\newtheorem{lemma}{\bf Lemma}[section]
\newtheorem{corollary}{\bf Corollary}[section]
\newtheorem{theorem}{\bf Theorem}[section]

\newtheorem{definition}{\bf Definition}[section]
\newenvironment{proof}{{\noindent \bf \em Proof:}}{\hfill$\square$}

\title{Spectral methods for nonlinear functionals and functional differential equations}

\begin{document}
\begin{frontmatter}

\author[ucsc]{Daniele Venturi\corref{correspondingAuthor}}
\ead{venturi@ucsc.edu}
\author[ucsc]{Alec Dektor}

\address[ucsc]{Department of Applied Mathematics, University of California Santa Cruz, Santa Cruz, CA 95064}

\cortext[correspondingAuthor]{Corresponding author}

\journal{arXiv}

\begin{abstract}
We present a rigorous convergence analysis for cylindrical approximations of nonlinear functionals, functional derivatives, and functional differential equations (FDEs). The purpose of this analysis is twofold: first, we prove that continuous nonlinear functionals, functional derivatives and FDEs can be approximated uniformly on any compact subset of a   {real Banach space admitting a basis} by high-dimensional multivariate functions and high-dimensional partial differential equations (PDEs), respectively. Second, we show that the convergence rate of such functional approximations can be exponential, depending on the regularity of the functional (in particular its Fr\'echet differentiability), and its domain. We also provide necessary and sufficient conditions for consistency, stability and convergence of cylindrical approximations to linear FDEs. These results open the possibility to utilize numerical techniques for high-dimensional systems such as deep neural networks and numerical tensor methods to approximate nonlinear functionals in terms of high-dimensional functions, and compute approximate solutions to FDEs by solving high-dimensional PDEs. Numerical examples are presented and discussed for prototype nonlinear functionals and for an initial value problem involving a linear FDE.


\end{abstract}

\end{frontmatter}
\section{Introduction}

\blfootnote{\noindent {\em 2020 Mathematics Subject Classification.}
46N40,  
35R15,  
47J05,  	
46G05,  	
65J15.  
}
A nonlinear functional is a map from a space
of functions into the real line or the complex plane.
Such map, which seems a rather abstract 
mathematical concept, plays a fundamental role 
in many areas of mathematical physics and applied 
sciences. In fact, nonlinear functionals were used, for example, 
by Wiener to describe Brownian motion 
mathematically \cite{Wiener66}, by Hohenberg 
and Kohn \cite{Hohenberg} to reduce 
the dimensionality of the Schr\"odinger equation 
in many-body quantum systems \cite{Parr,LinLin2019}, 
by Hopf to describe the statistical properties of turbulence
\cite{Hopf,Monin2,Alankus},  and by Bogoliubov 
to model systems of interacting bosons in 
superfluid liquid helium \cite{Bogoliubov,Seiringer}. 
Applications of nonlinear functionals
to other areas of mathematical physics can be found 
in \cite{Klyatskin1,Daniele_JMathPhys,Amit,Fox,Jensen,Martin}. 

Nonlinear functionals have also appeared in 
evolution equations known as functional differential 
equations (FDEs) \cite{venturi2018numerical}.
A classical example in fluid dynamics is
the Hopf equation \cite{Hopf,Ohkitani,Rosen_1971}  
\begin{equation}
\frac{\partial \Phi([ \theta],t)}{\partial t}=
\sum_{k=1}^3\int_V\theta_k( x)\left(i \sum_{j=1}^3\frac{\partial }{\partial x_j}
\frac{\delta^2 \Phi([ \theta],t)}{\delta \theta_k( x)\delta\theta_j( x)}
+\nu \nabla^2\frac{\delta \Phi([ \theta],t)}{\delta \theta_k( x)}\right)d x, 
\label{hopfns}
\end{equation}
which governs the dynamics of the 
characteristic functional 
\begin{equation}
\Phi([ \theta],t)=\mathbb{E}\left\{\exp \left(i\int_{V}
 u (x,t)\cdot \theta(x)d x\right)\right\}.
\label{HcF}
\end{equation}
Here, $u(x,t)$ represents a stochastic solution 
to the Navier-Stokes equation corresponding 
to a random initial state, and $\mathbb{E}\{\cdot\}$ 
is the expectation over the probability measure of 
such random initial state. Remarkably, 
the complex-valued nonlinear functional \eqref{HcF} 
encodes all statistical information of the stochastic 
solution to the Navier Stokes equation. For this reason, 
equation \eqref{hopfns} was deemed 
by Monin and Yaglom (\cite[Ch. 10]{Monin2}) 
to be ``the most compact formulation of the general 
turbulence problem'', which is the problem of 
determining the statistical properties of the velocity 
and the pressure fields of the Navier-Stokes equations 
given statistical information on the initial state\footnote{
In equations \eqref{hopfns}-\eqref{HcF},
$V\subseteq \mathbb{R}^3$ is a periodic box, 
$\theta(x)=(\theta_1( x), \theta_2( x), \theta_3( x))$ 
is a vector-valued (divergence-free) function, 
and $\delta /\delta \theta_j(x)$ denotes the first-order 
functional derivative \cite{Hanggi4}.}. 
Another well-known example of functional differential 
equation is the Schwinger-Dyson equation of quantum 
field theory \cite{Peskin,Justin}. Such equation 
describes the dynamics of the generating functional 
of the Green functions of a quantum field theory, 
allowing us to propagate field interactions 
in a perturbation setting (e.g., with Feynman 
diagrams), or in a strong coupling regime. 
The Schwinger-Dyson functional formalism is also useful 
in studying statistical dynamics of classical 
systems described in terms of stochastic ordinary or 
partial differential equations\footnote{
The solution to a stochastic ordinary or 
partial differential equation is a nonlinear 
functional of the forcing terms, initial 
condition and boundary conditions. 
Effective methods to represent such functional 
dependence are based on 
polynomial chaos expansions
\cite{db1,Ernst,DanieleWick,Venturi1,Venturi_JFM2}, 
probabilistic collocation methods 
\cite{Foo1,DoostanOwhadi_2011,Db_book}, and deep neural networks \cite{Raissi1,Zhu2019}. 
Other techniques rely on a reformulation of 
the problem in terms of kinetic equations 
\cite{Venturi_JCP2013,Heyrim2017,parr_tensor_BGK},
or hierarchies of kinetic equations 
\cite{Brennan2018,Venturi_MZ,HeyrimPRS_2014}.}  
\cite{Jensen,Martin,Phythian}.
More recently, FDEs appeared 
in mean field games \cite{Carmona}, 
and mean field optimal control \cite{Ruthotto2020,Weinan2019}. 
Mean field games are optimization problems involving 
a very large (potentially infinite) number of interacting 
players. In some cases, it is possible to reformulate such 
optimization problems in terms of a nonlinear 
Hamilton-Jacobi FDE in probability density space.
The standard form of such equation is \cite{Osher2019_1}
\begin{equation}
\frac{\partial F([\rho],t)}{\partial t}+
\mathcal{H}\left([\rho],\left[\frac{\delta F([\rho],t)}{\delta \rho( x)}\right]\right)=0,
\qquad F([\rho],0)=F_0([\rho]),
\label{FDE11}
\end{equation}
where $\rho(x)$ is a $n$-dimensional probability density 
function supported on $\Omega\subseteq \mathbb{R}^n$, 
$\delta F /\delta \rho(x)$ is the first-order functional 
derivative of $F$ relative to $\rho(x)$, and $\mathcal{H}$ 
is the Hamilton functional
\begin{equation}
\mathcal{H}\left([\rho],\left[\frac{\delta F([\rho],t)}{\delta \rho( x)}\right]\right) = 
\int_{\Omega} {\Psi}\left(x,\nabla \frac{\delta F([\rho],t)}
{\delta \rho(x)} \right)\rho(x)dx + G([\rho]).
\end{equation}
Here, $\Psi$ is a Hamilton function, 
and $G([\rho])$ is an interaction potential.
More general FDEs of the type \eqref{FDE11} have 
been recently derived in the context of unnormalized 
optimal PDF transport \cite{Osher2019}.
Mean field theory is also useful in optimal feedback control 
of nonlinear stochastic dynamical systems, and 
in deep learning. For instance, recent work of W. E 
and collaborators \cite{Weinan2019} laid the 
mathematical foundations of the population risk 
minimization problem in deep learning as 
a mean-field optimal control problem. Such 
mean-field optimal control problem yields 
a generalized version of the Hamilton-Jacobi-Bellman  
equation in a Wasserstein space, 
which is a nonlinear FDE (see Eq. (20) in \cite{Weinan2019}). 
%
%

{Computing accurate approximations 
of the solution to FDEs such as \eqref{hopfns} 
or \eqref{FDE11} is a long-standing  
problem in mathematical physics. In a recent Physics 
Report \cite{venturi2018numerical}, we reviewed  
state-of-the-art methods to approximate 
nonlinear functionals and FDEs.
In particular, we discussed 
an approximation method, known as ``cylindrical 
approximation'', in which nonlinear functionals 
and FDEs defined on function spaces admitting 
a basis are approximated by multivariate 
functions and multivariate partial 
differential equations (PDEs), respectively. 
The idea is, if a function 
space admits a basis, then any function in the 
space can be represented uniquely by 
projection coefficients onto the basis. 
Accordingly, nonlinear functionals  
defined on such a function space can be 
represented as multivariate functions 
of the coefficients. 
%
%
The objective of this paper is to
provide a rigorous mathematical 
foundation for cylindrical approximations 
to nonlinear functionals, functional
derivatives, and FDEs defined on
Banach spaces admitting a basis.
The purpose of this analysis is twofold: 
first, we prove that cylindrical 
approximations converge uniformly on 
compact subsets of real Banach spaces 
admitting a basis.
Second, we prove that the convergence 
rate can be exponential in the number of 
projection coefficients.
We also provide necessary and sufficient 
conditions for consistency, stability and 
convergence of cylindrical approximations 
to FDEs based on the Trotter-Kato approximation 
theorem \cite{Guidetti2004,engel1999one}.}

This paper is organized as follows. In section
\ref{sec:NonlinearFunctionals} we briefly review 
the theory of nonlinear functionals defined on 
a Banach space, and recall the notions of continuity, 
compactness and differentiability. In section 
\ref{sec:functionalonmetric} we specialize these 
concepts to nonlinear functionals defined on a real 
separable Hilbert space $H$.
In section \ref{sec:finiteDimapprox} we introduce 
{cylindrical approximations} of nonlinear 
functionals and functional derivatives defined on a 
Hilbert space. Uniform convergence for both approximations 
is established in section \ref{sec:F} and section 
\ref{sec:FD}, provided the functional (or 
functional derivative) is defined on 
a compact subset of $H$.  We also show 
that cylindrical approximations
can converge exponentially fast for 
Fr\'echet differentiable functionals.
In section \ref{sec:FDE} we develop a 
self-consistent convergence analysis of 
cylindrical approximations to linear FDEs
in compact subsets of real separable Hilbert 
{spaces}. 
{In section \ref{sec:Banach} we 
outline the extension of the functional 
approximation theory we developed in 
Hilbert spaces to compact subsets of real Banach 
spaces admitting a basis.}
In section \ref{sec:numerics} we 
{provide numerical examples 
demonstrating convergence of cylindrical 
approximations of nonlinear 
functionals and a linear FDE. In particular, 
we study the Hopf equation corresponding 
to a linear advection problem evolving from 
a random initial state}.
The main findings are summarized in section \ref{sec:Conclusion}.
We also include two brief Appendices 
where we discuss cylindrical approximations 
of functional integrals in real separable 
Hilbert spaces, and the notion of 
distance between function spaces.

\section{Nonlinear functionals in Banach spaces}
\label{sec:NonlinearFunctionals}

Let $X$ be a Banach space. A nonlinear functional 
on $X$ is a map $F$ from $X$ into a field 
$\mathbb{F}$. In this paper, $\mathbb{F}$ will either be  
the real line ($\mathbb{R}$) or 
the complex plane ($\mathbb{C}$).  In general, the 
functional $F$ does not operate on the entire Banach 
space $X$ but rather on a subset set of $X$, which 
we denote as $D(F) \subseteq X$ (domain of the functional) 
\begin{equation}
F: D(F)\subseteq X \to \mathbb{F}. 
\label{nonlinF}
\end{equation}
As an example, consider
\begin{equation}
F([\theta])=\int_0^1 x^3 e^{\theta(x)+\theta'(x)}dx, \qquad 
\theta\in D(F)=C^{(1)}([0,1]),
\label{functional1}
\end{equation}
where $C^{(1)}([0,1])$ is the space of 
continuously differentiable real-valued functions 
defined on $[0,1]$. The map \eqref{functional1} associates 
to each function $\theta\in C^{(1)}([0,1])$ the real number $F([\theta])$. 
Analysis of nonlinear functionals in Banach spaces 
is a well developed subject \cite{Vainberg,Nashed,Schwartz1964,HunterNachtergaele2001,Folland}.
In particular, classical definitions of continuity and 
differentiability that hold for real-valued functions 
can be extended to  functionals. For instance, 
\begin{definition}
\label{def:ptwise:continuity}
{\bf (Pointwise continuity of functionals)}
A nonlinear functional 
$F: D(F)\subseteq X \to \mathbb{F}$ 
is continuous at a point 
$\theta\in D(F)$ if for any Cauchy sequence  
$\{\theta_1, \theta_2, \ldots \}$ in $D(F)$ 
converging to $\theta$ (in the metric of $X$) 
we have that the sequence 
$\{F([\theta_1]), F([\theta_1]), \ldots\}$ 
converges to $F([\theta])$ (in the metric of $\mathbb{F}$), i.e., 
\begin{equation}
\lim_{n\rightarrow \infty}
\left\|\theta_n(x)-\theta(x)\right\|_X=0 \quad \Rightarrow \quad \lim_{n\rightarrow \infty}
\left|F([\theta_n])-F([\theta])\right|= 0.
\label{continuity}
\end{equation}
\end{definition}
\vs
\noindent
{\em Example 1:} The functional \eqref{functional1} is 
continuous at every point $\theta\in C^{(1)}([0,1])$ 
relative to the norm 
\begin{equation}
\| \theta\| = \sup_{x\in[0,1]}|\theta(x)|+
\sup_{x\in[0,1]}|\theta'(x)|.
\end{equation}

\begin{definition}
\label{def:uniform_continuity}
{\bf (Uniform continuity of functionals)} The functional $F: D(F)\subseteq X \to \mathbb{F}$ is 
said to be {uniformly continuous} on the domain 
$D(F)$ if for every $\epsilon>0$ there exists $\delta>0$ 
such that the inequality $| F([\theta_1])-F([\theta_2])|<\epsilon$ holds
for all points $\theta_1,\theta_2\in D(F)$ satisfying 
$\left\|\theta_1-\theta_2\right\|_X\leq \delta$.
\end{definition}

\noindent
Note that the definition of continuity and uniform 
continuity of a functional depends on how we 
measure the distance between elements of Banach 
space $X$.

\begin{definition}
\label{def:complete_continuity}
{\bf (Compactness and complete continuity of functionals)} 
The functional $F: D(F)\subseteq X \to \mathbb{F}$ 
is said to be compact on the domain $D(F)$ 
if it maps every bounded subset of $D(F)$ into 
a pre-compact subset set 
of $\mathbb{F}$, i.e., a subset whose 
closure is compact. The functional $F([\theta])$ 
is called {completely continuous} on $D(F)$ if 
it is continuous and compact. 
\end{definition}
It is clear that a continuous functional $F([\theta])$ 
is completely continuous if and only if for every bounded 
sequence $\{\theta_n\}$ in $D(F)$ we have that 
the sequence $\{F([\theta_n])\}$ has a 
convergent sub-sequence. As is well-known, continuous 
functions defined on a closed and bounded subset 
of $\mathbb{R}^n$ are always 
uniformly continuous and bounded. This is not the case with 
functionals defined on Banach spaces. In fact, 
uniform continuity of a functional $F([\theta])$ 
on a closed and bounded set $S\subseteq D(F)$, 
say the unit sphere $S=\{\theta\in C^{(1)}([0,1]): 
\left\|\theta\right\|\leq 1\}$\footnote{Recall 
that a closed sphere $S$  in a Banach space $X$ 
is not compact as not every sequence of 
elements in the sphere has a convergent 
sub-sequence with limit in $S$.}, 
is not sufficient to guarantee that the 
functional is bounded \cite[p. 18]{Vainberg}.
However, if the functional $F([\theta])$ is uniformly 
continuous on $D(F)$ then it maps compact sets into 
compact sets. Moreover, if $F$ is completely continuous 
on a bounded set $K\subseteq D(F)$ (open or closed) 
then $F$ is bounded on $K$. This is obvious since
the definition of complete continuity of $F$ implies that 
bounded sets are mapped into relatively 
compact sets, which are bounded. 

\vs
\noindent
{\em Example 2:} Consider the nonlinear functional 
\begin{equation}
F([\theta]) = \sin(\theta(\pi))
\label{Fo}
\end{equation}
in the Banach space of Lipschitz continuous 
periodic functions in $[0,2\pi]$, 
\begin{equation}
D(F)= C_{\text{Lip}}^{(0)}([0,2\pi]).
\label{topospace}
\end{equation}
As is well known , 
the Fourier series of any element 
$\theta\in D(F)$ defines a sequence of partial sums 
$\{\theta_m\}\in D(F)$ that converge uniformly 
to $\theta$ (see \cite{Jackson}). Thanks to such uniform
convergence result, we have 
\begin{align}
 \lim_{m\rightarrow \infty}
 \left| F([\theta_m])-F([\theta])\right|&=
\lim_{m\rightarrow \infty}
\left| \sin(\theta_m(\pi))-\sin(\theta(\pi)) \right|\nonumber\\
&\leq 
\lim_{m\rightarrow \infty} \sup_{x\in[0,2\pi]}
\left| \sin(\theta_m(x))-\sin(\theta(x))\right|\nonumber \\
&= 0,
\end{align}
where the last equality follows from \cite[Theorem 10]{Bartle}.
Hence, the functional \eqref{Fo} is continuous on 
$C^{(0)}_{\text{Lip}}([0,2\pi])$. Moreover, $F$ sends any 
bounded subset of such function space into a 
pre-compact subset of 
the real line, and therefore $F$ is 
completely continuous.

\subsection{Differentials and derivatives of nonlinear functionals}

Let us consider a nonlinear functional $F: D(F)\subseteq X \to \mathbb{F}$, where $D(F)$ is an open set. We say that 
$F$ is G\^ateaux differentiable at a point $\theta\in D(F)$ 
if the limit  
\begin{equation}
dF_\eta([\theta]) = 
\lim_{\epsilon\rightarrow 0}\frac{F([\theta+\epsilon \eta])
-F([\theta])}{\epsilon}
\label{gateaux}
\end{equation}
exists and is finite for all $\eta\in D(F)$. 
The quantity $dF_\eta([\theta])$ is known 
as {\em G\^ateaux differential} of $F$ in the 
direction of $\eta$ \cite{Vainberg,Schwartz1964}. 
Under rather mild conditions (see, e.g., \cite[p. 37]{Vainberg}) 
such differential can be represented as a linear 
operator acting on $\eta$ \cite{Nashed,Daniele_JMathPhys}. 
Such linear operator is known as the 
{\em G\^ateaux derivative} of $F$ at $\theta$ 
and and it will be denoted by $F'([\theta])$
 \begin{equation}
 dF_\eta ([\theta]) = F'([\theta]) \eta.
 \label{derivativeD}
 \end{equation}
The {\em Fr\'echet differential}, on the other hand, 
is defined as the term $dF([\theta],\eta)$ in the series 
expansion
\begin{equation}
F([\theta+\eta])=F([\theta])+dF([\theta], \eta) + 
R([\theta],[\eta]),\qquad \lim_{\left\|\eta\right\|\rightarrow 0} \frac{|R([\theta],[\eta])|}{\left\|\eta\right\|}=0.
\end{equation}
It is well-known that if $F([\theta])$ 
has a continuous G\^ateaux derivative on $D(F)$, 
then $F$ is Fr\'echet differentiable on $D(F)$, and 
these two derivatives coincide \cite[p. 41]{Vainberg}. 
In this paper, we consider nonlinear 
functionals $F$ that are continuously G\^ateaux 
differentiable in $D(F)$. Hence, we will not need to 
distinguish between Fr\'echet and G\^ateaux 
derivatives. 

There has been significant research activity on 
obtaining the minimal conditions under which a nonlinear 
functional is G\^ateaux or Fr\'echet differentiable. 
It turns out that there are reasonably satisfactory 
results on G\^ateaux differentiablility of Lipschitz 
functionals. For instance, 
\begin{theorem}
\label{prop:gateaux_lipschitz}
{(\bf G\^ateaux differentiablility of Lipschitz 
functionals \cite{Mankiewicz,Aronszajn,Lindenstrauss}) } 
Let $X$ be a separable Banach space. Then every real-
or complex-valued Lipschitz functional $F$ from an 
open set $D(F)\subseteq X$  is G\^ateaux differentiable 
outside a Gauss-null set\footnote{A Gauss-null set is a 
Borel set $A\subseteq X$ such that $\mu(A)=0$ for every 
 non-degenerate Gaussian measure $\mu$ on $X$.}. 
\end{theorem}
Results on Fr\'echet differentiability are more rare, 
and usually much harder to prove \cite{Lindenstrauss}.
For instance, we have
\begin{theorem}
\label{thm:frechet_lipschitz}
{\bf (Fr\'echet derivatives of Lipschitz functionals \cite{Preiss})} 
Let $K$ be a compact subset of a Hilbert space $H$. 
If $F([\theta])$ is real-valued and locally Lipschitz on $K$ 
then $F$ is Fr\'echet differentiable on a dense subset of $K$.  
\end{theorem}
Note that Theorem \ref{thm:frechet_lipschitz} 
does not imply that $F$ is Fr\'echet differentiable 
everywhere on $K$. 
On the other hand, a continuously differentiable 
functional $F$, which is also completely continuous 
in the sense of Definition \ref{def:complete_continuity}, 
has completely continuous G\^ateaux and Fr\'echet derivatives (\cite[p.51]{Vainberg}). 
As we will see in section \ref{sec:functionalonmetric}, 
continuously differentiable nonlinear functionals 
on compact metric spaces are also compact, and have 
compact Fr\'echet derivative. 
We emphasize that it is also possible to 
define G\^ateaux and Fr\'echet 
differentiability directly in terms of bounded 
linear operators. For instance, Lindenstrauss and 
Preiss \cite[p. 258]{Lindenstrauss} define $F$ as
G\^ateaux differentiable if there exists 
a bounded linear operator 
$F':D(F)\rightarrow \mathbb{F}$ such that 
for every $\eta\in D(F)$
\begin{equation}
F'([\theta])\eta = 
\lim_{\epsilon\rightarrow 0}\frac{F([\theta+\epsilon \eta])
-F([\theta])}{\epsilon}.
\nonumber 
\end{equation}
Clearly, this definition is more strict than 
\eqref{gateaux}-\eqref{derivativeD}, 
since it does not allow for unbounded derivative 
operators $F'([\theta])$.

In the context of nonlinear functionals 
defined on spaces of functions, 
it is convenient to define another type of 
functional derivative, namely 
\begin{equation}
\frac{\delta F([\theta]) }{\delta \theta(x)} = 
\lim_{\epsilon\rightarrow 0}
\frac{F([\theta(y)+\epsilon \delta (x-y)])
-F([\theta(y)])}{\epsilon},
\label{first_order_FD}
\end{equation}
provided the limit exists. 
The quantity $\delta F([\theta])/\delta \theta(x)$, is 
known as {\em first-order functional derivative} 
of $F([\theta])$ with respect to $\theta(x)$ 
\cite[p. 309]{Hanggi2}). Functional 
derivatives are used extensively in many areas 
mathematical physics, e.g., in stochastic dynamics 
\cite{Hanggi4,Venturi_PRS,Fox,Jensen}, 
turbulence modeling \cite{Monin2,Ohkitani,Dopazo}, 
and quantum  field theory \cite{Justin,Parr}.

If the Fr\'echet derivative $F'([\theta])$ admits an 
integral representation\footnote{
Conditions under which linear operators between 
spaces of functions admit an integral 
representation were investigated in 
\cite{Bukhvalov,Schachermayer,Schep,Cilia}.},
then it is possible to establish a one-to-one correspondence 
between $\delta F([\theta])/\delta \theta(x)$ 
and $F'([\theta])$. 
For instance, if $\ell(\eta) = F'([\theta])\eta$ is a bounded linear 
functional in a Hilbert space $H$, and $F'([\theta])$ 
is continuous in $D(F)\subseteq H$ then the Riesz 
representation theorem guarantees that there exists a 
unique element $\delta F([\theta])/\delta \theta(x)\in H$
such that 
\begin{equation}
F'([\theta])\eta =\left( 
\frac{\delta F([\theta]) }{\delta \theta(x)},\eta(x) 
\right)_H,  \qquad \forall \eta\in H, \qquad \forall \theta\in D(F). 
\label{first_order_L}
\end{equation}
Here $(\cdot,\cdot)_H$ denotes the inner product in $H$. 
As we will see in section \ref{sec:functionalonmetric}, the Fr\'echet 
derivative of a continuous nonlinear functional defined 
on a compact subset of a real separable Hilbert space 
is a completely continuous linear operator, i.e., 
continuous and compact. In this case, the Riesz 
representation \eqref{first_order_L} holds and we 
have a one-to-one correspondence between $F'([\theta])$ and 
$\delta F([\theta])/\delta \theta(x)$.

More generally, if $\Omega$ is a locally compact 
Hausdorff space, e.g., an open or closed subset 
of $\mathbb{R}^n$, and $F'([\theta])$
is a bounded linear operator from 
$C_c^{(0)}(\Omega)$ (space of 
compactly supported continuous functions on $\Omega$) into 
$\mathbb{R}$ then there exists a unique finite regular signed 
measure\footnote{A signed measure is a generalization of the 
concept of measure by allowing it to have negative values.} 
$\mu([\theta],x)$ on the Borel subsets of $\Omega$ such 
that
\begin{equation}
F'([\theta])\eta=\int_{\Omega} \eta(x) d\mu([\theta],x), \qquad \eta \in C_c^{(0)}(\Omega),
\label{Haus}
\end{equation} 
(see \cite[p. 324]{Semadeni}, \cite[p. 4]{Diestel} or \cite{Zakharov,Folland}).
In addition, if $\mu([\theta],x)$ is absolutely continuous with 
respect to $x$ then there exists a Radon-Nikodym derivative, 
i.e., a functional density $\delta F([\theta])/\delta \theta(x)$, such that 
\begin{equation}
 d\mu([\theta],x) = \frac{\delta F([\theta]) }{\delta \theta(x)} dx.
\end{equation}
Under these conditions, the Fr\'echet derivative of $F$ admits 
the Lebesgue integral representation 
\begin{equation}
F'([\theta])\eta=\left( \frac{\delta F([\theta]) }{\delta \theta(x)},\eta(x)\right)_{L^2(\Omega)}.
\label{FGD}
\end{equation} 
We emphasize that \eqref{first_order_L} and \eqref{FGD}
can be considered as infinite-dimensional generalizations
of the concept of directional derivative 
of a multivariate function $f(x)$, in which the 
dot product between the gradient $\nabla f$ and 
a vector $\eta$ is now replaced by the 
inner product of $\delta F([\theta])/\delta \theta(x)$
and $\eta(x)$.
By analogy, the quantity $\delta F([\theta])/\delta \theta(x)$ 
can be thought of as an infinite-dimensional gradient. 
Note that such gradient is a nonlinear 
functional of $\theta$ and a function of $x$. Higher-order 
Fr\'echet and functional derivatives can be defined 
similarly \cite{venturi2018numerical,Hanggi4}.

\vs
\noindent
{\em Example 1:}
The Fr\'echet derivative of the nonlinear functional 
\eqref{Fo} is obtained as
\begin{equation}
F'([\theta])\eta = \cos(\theta(\pi))\eta(\pi),
\label{FrD}
\end{equation}
where both $\theta$ and $\eta$ are 
in the space \eqref{topospace} of Lipschitz 
continuous periodic functions in $[0,2\pi]$. 
Clearly, equation \eqref{FrD} can be written as 
\begin{equation}
F'([\theta])\eta = \int_0^{2\pi} \cos(\theta(x))\delta(x-\pi) \eta(x) dx.
\label{r5}
\end{equation}
From this expression we see that the signed 
measure $d\mu([\theta],x)$ appearing in equation
\eqref{Haus} in this case has a density, which 
coincides with the 
first-order functional derivative
\begin{equation}
\frac{\delta F([\theta])}{\delta \theta(x)} = \cos(\theta(x))\delta(x-\pi).
\label{fffd}
\end{equation}
Such derivative is a distribution in $x$ and 
a continuous functional of $\theta$. The 
Fr\'echet differential \eqref{FrD} is a linear 
functional in $\eta$, bounded in the $C^{(0)}$ 
norm.
In fact,  
\begin{equation}
\left|\cos(\theta(\pi))\eta(\pi)\right|\leq \left|\cos(\theta(\pi))\right| 
\sup_{x\in[0,2\pi]}\left|\eta(x)\right|.
\end{equation}
However, such functional is unbounded in 
$L^2([0,2\pi])$, and therefore \eqref{r5} is not 
based on the Riesz representation 
\eqref{first_order_L}, but rather on \eqref{FGD}. 
To show this, we just need to prove that 
\eqref{FrD} admits an integral representation 
with kernel that is not in $L^2([0,2\pi])$. To this end, 
let us represent $\eta$ relative 
to any orthonormal trigonometric basis 
$\{\varphi_0,\varphi_1,\ldots\}$ of $L_p^2([0,2\pi])$ 
(space of square integrable periodic functions in $[0,2\pi]$)
\begin{equation}
\eta(x)=\sum_{k=0}^\infty a_k\varphi_k(x), \qquad a_k=(\eta,\varphi_k)_{L_p^2([0,2\pi])}.
\label{sereta}
\end{equation}
The series \eqref{sereta} converges in the 
$L_p^2([0,2\pi])$ sense, and also pointwise 
since $\eta$ is continuous \cite{Jackson}. 
A substitution of \eqref{sereta} into  \eqref{FrD} yields 
\begin{equation}
\cos(\theta(\pi))\eta(\pi) = \left(\cos(\theta(\pi))  v(x),\eta(x)\right)_{L_p^2([0,2\pi])}, \qquad 
v(x)=\sum_{k=0}^\infty \varphi_k(\pi)\varphi(x).
\label{le2}
\end{equation}
It straightforward to show that
\begin{equation}
\sum_{k=0}^\infty \varphi_k(\pi)^2 = \infty,
\end{equation}
and therefore the function $v(x)$ in \eqref{le2} is not 
an element of $L_p^2([0,2\pi])$. Indeed, $v(x)$ is the 
trigonometric series expansion of the Dirac delta 
function $\delta (x-\pi)$, which is not in $L_p^2([0,2\pi])$.

\section{Nonlinear functionals defined on compact subsets of real separable Hilbert spaces} 
\label{sec:functionalonmetric}

In this section we study  the mathematical 
properties of nonlinear functionals defined on 
compact subsets of real separable Hilbert spaces. 
As we will see in subsequent sections, this is 
a very important class of functionals which allows us 
to build an effective approximation theory 
based on orthogonal projections.

Before we present such theory, let us 
briefly review the notion of bounded, 
closed, compact and pre-compact 
sets. Consider a metric space of 
functions $X$, e.g., a Hilbert or a Banach space, 
and a subset $K\subseteq X$. 
We say that $K$ is {\em bounded} if for all $\theta \in K$ 
we have $\left\|\theta\right\|_X \leq M$, where $M$ is a finite 
real number, and $\|\cdot \|_X$ is the norm in $X$. 
The set $K$ is said to be {\em closed} if any 
convergent sequence in $K$ has a limit in $K$. 
An example of a closed 
and bounded subset of the Hilbert space 
$L^2([0,1])$ is the unit sphere 
$S=\{\theta \in L^2([0,1]): \left\|\theta\right\|_{L^2([0,1])} \leq 1\}$. 
We say that the subset $K\subseteq X$ is 
{\em compact} if every open cover of $K$ 
has a finite subcover. 

There are several equivalent 
characterizations of compactness in metric 
spaces. For instance, a subset $K$ of a 
metric space $X$ is compact 
if and only if every sequence in $K$ has a 
bounded subsequence whose limit is in $K$ 
\cite[\S 1.7]{HunterNachtergaele2001}.
The set $K$ is said to be {\em pre-compact}
if its closure is compact, meaning that every 
sequence in $K$ has a convergent 
sub-sequence whose limit is in $X$ (not in $K$). 
Closed and bounded function spaces are not necessarily 
compact, since we can define sequences that do not have 
convergent sub-sequences. An example is the the unit
sphere $S$ mentioned above. 
On the other hand, a compact set 
is always bounded and closed.
A useful characterization of pre-compactness in 
real separable Hilbert spaces is the following:
\begin{theorem} 
\label{compactSets}
{\bf (Compact subsets of real separable Hilbert spaces)} 
A subset $E$ of a real separable Hilbert space $H$ 
is pre-compact if and only if it is bounded, closed, and 
for any (one) orthonormal basis $\{\varphi_1,\varphi_2,\ldots\}$ 
of $H$, and any $\epsilon>0$ there exists a natural number $m$ 
such that\footnote{The condition \eqref{equi} is known as 
{\em equi-small tail condition} with respect to any 
orthonormal basis of $H$. It states that 
if we use enough basis elements, then we can bound the 
complementary energy (sum of the neglected Fourier 
amplitudes) associated with the series expansion of  
of {any} $\theta\in E$.}
\begin{equation}
\sum_{k=m+1}^\infty |(\theta,\varphi_k)_H|^2 \leq \epsilon,
\label{equi}
\end{equation}
for all $\theta \in E$. The closure of $B$ in $H$, 
which we denote as $\overline{E}$, is compact.
\end{theorem}
The proof of Theorem \ref{compactSets} can be 
found in \cite[p. 76]{Melrose} (Proposition 3.8). 
Note that neither pre-compact nor compact subsets 
of a vector space can be vector spaces, since 
pre-compactness implies boundedness. Hereafter 
we provide two simple examples of compact 
subsets of real separable Hilbert spaces.

\vs
\noindent
{\em Example 1:} 
Let $H_p^{s}([0,2\pi])$ be the Sobolev space of weakly 
differentiable (up to degree $s\geq 1$) periodic 
functions in $[0,2\pi]$.  The set 
\begin{equation}
K = \overline{\left\{\theta\in H_p^{s}([0,2\pi])\,:\, \left\|\frac{d^q\theta}{dx^q}\right\|_{L_p^{2}([0,2\pi])}\leq\rho \right\}}\subseteq  
L_p^{2}([0,2\pi]) \qquad 1\leq q\leq s,
\label{Sobolevsphere0}
\end{equation}
where $L_p^{2}([0,2\pi])$ is the 
Lebesgue space of periodic functions 
in $[0,2\pi]$ and $\rho>0$ is 
the radius of the Sobolev sphere, is a compact subset 
of $L_p^{2}([0,2\pi])$. Indeed, by expanding an 
arbitrary element $\theta \in K$ in a Fourier series 
we obtain, for any $1\leq q\leq s$ 
(see \cite[p. 35]{spectral})
\begin{equation}
\sum_{|k|>m}|(\theta,e^{ikx})_{L_p^2([0,2\pi])}|^2 \leq 
\frac{C_1}{m^{2q}}\left\| \frac{d^q\theta}{dx^q}\right\|^2_{L_p^2([0,2\pi])}\leq \frac{C_1\rho^2}{m^{2q}}  \qquad \forall \theta \in K.
\label{rhsL2}
\end{equation} 
At this point it is clear that for any given 
$\epsilon>0$ there exists a natural number $m$ 
such that the right hand side of \eqref{rhsL2} 
can be made smaller than $\epsilon$ for 
any $\theta\in K$. In other words, the equi-small 
tail condition \eqref{equi} is satisfied. Moreover, 
in \eqref{Sobolevsphere0} we take the closure 
of the Sobolev sphere in $L_p^2([0,2\pi])$, which makes 
$K$ is a compact subset of $L_p^2([0,2\pi])$.

\vs
\noindent
{\em Example 2:} A closed sphere with radius $\rho$ 
in $H_w^s([-1,1])$  (weighted Sobolev space of degree $s$),
is a pre-compact subset of $L_w^2([-1,1])$ (weighted 
Lebesgue space of functions in $[-1,1]$). Hence,  
\begin{equation}
K=\overline{\left\{\theta\in H_w^s([-1,1]) \,:\, 
\left\| \theta\right\|_{H_w^s([-1,1])}\leq \rho\right\}}
\subseteq L_w^2([-1,1])
\label{Sobolevsphere}
\end{equation} 
is a compact subset of $L_w^2([-1,1])$. This claim is based 
on the following well-known spectral 
convergence result \cite[p. 109]{spectral},
\begin{equation}
\sum_{k=m+1}^{\infty}|(\theta,\varphi_k)_{L^2_w([-1,1])}|^2 \leq 
\frac{C_2}{m^{s}}\left\| \theta\right\|^2_{H_w^s([-1,1])},
\label{33}
\end{equation}
where $\varphi_k(x)$  here are ultra-spherical 
polynomials. By combining \eqref{33} with 
\eqref{Sobolevsphere}, we see that the equi-small 
tail condition 
\begin{equation}
\sum_{k=m+1}^{\infty}|(\theta,\varphi_k)_{L^2_w([-1,1])}|^2 \leq 
\frac{C_2\rho^2}{m^{s}}
\label{34}
\end{equation}
is satisfied $\forall\theta\in K$. 

\vs 

The compact subsets we discussed in 
Example 1 and Example 2 are particular instances 
of a general compact embedding result known as 
Rellich-Kondrachov theorem \cite[\S 6]{Adams}. 
Such theorem states that the Sobolev space 
of functions $W^{k,p}(\Omega)$ defined on a 
compact subset $\Omega\subseteq\mathbb{R}^n$ 
with differentiable boundary is compactly 
embedded in $W^{l,q}(\Omega)$, provided $k>l$ and 
$k-p/n>l-q/n$. 
This means that there exists a compact linear operator 
$T:W^{k,p}(\Omega)\rightarrow W^{l,q}(\Omega)$ 
that maps bounded subsets of $W^{k,p}(\Omega)$ 
into pre-compact subsets of $W^{l,q}(\Omega)$. 
In Example 1 and Example 2 we have that $H^s = W^{s,2}$ 
and $L^2= W^{0,2}$  Hence, the Rellich-Kondrachov 
embedding in this case reduces to the statement 
that the Sobolev space $H^s$ is compactly embedded 
in $L^2$ for all $s>0$.

Continuous nonlinear functionals defined on a 
compact subset $K$ of a metric space have 
nice mathematical properties. First of all, 
they are bounded since they map compact sets into 
a closed and bounded subset of $\mathbb{R}$ or $\mathbb{C}$. 
Moreover, by the Heine-Cantor Theorem we have that 
any continuous functional defined on a compact set 
is necessarily uniformly continuous and 
bounded \cite{HunterNachtergaele2001}.
If the functional is real-valued this means that 
the maximum and the minimum are attained 
at points within $K$. We also recall that closed subsets 
of compact sets are necessarily compact. Hence, a continuous 
functional on $K$ maps any 
closed subset of $K$ into a closed and bounded 
subset of $\mathbb{R}$ or $\mathbb{C}$.
Such functional is necessarily 
compact\footnote{We recall that a 
compact nonlinear functional is a continuous 
functional that maps bounded sets into 
pre-compact (relatively compact) 
sets.}, i.e., completely continuous 
(see Definition \ref{def:complete_continuity}).

\subsection{Fr\'echet and functional derivatives}
Next, we show that the Fr\'echet derivative of 
a continuous nonlinear functional defined 
on a compact subset of a real separable Hilbert 
space is a compact linear operator.
\begin{lemma}
\label{COMPFRE}
{\bf (Compactness of the Fr\'echet derivative)} 
Let $K$ be a compact subset of a real separable 
Hilbert space $H$,
and let $F([\theta])$ be a continuous 
real- or complex-valued 
functional on $H$. If the Fr\'echet 
derivative $F'([\theta^*])$ 
exists at $\theta^*\in K$, then it is a compact linear 
operator\footnote{If the Fr\'echet derivative 
is {\em defined} to be a bounded linear operator, e.g., as in  
\cite{Lindenstrauss}, then Lemma \ref{COMPFRE} is 
trivial since any bounded linear operator from $H$
to $\mathbb{F}$ is compact (the dimension of the 
co-domain being finite).}.
\end{lemma}
\begin{proof}
Continuous functionals on compact metric spaces 
are necessarily completely continuous (see Definition \ref{def:complete_continuity}). To prove the Lemma 
we proceed by contradiction. To this end, 
suppose that $F'([\theta^*])$ is not compact. Then it is 
possible to find $\epsilon > 0$ and a sequence 
$\{\theta_k\}\in K\subseteq H$ such that $\|\theta_k\|_H\leq 1$ and 
\begin{equation}
\left| F'([\theta^*])\theta_k -F'([\theta^*])\theta_j  \right| \geq \epsilon
\end{equation} 
for all $k\neq j$. By definition of Fr\'echet derivative 
at $\theta^*$ 
we have  
\begin{equation}
\left| F([\theta^*+\eta]) - F([\theta^*]) - 
F'([\theta^*])\eta)\right|= 
{o(\left\|\eta\right\|_H)},
\end{equation}
for all $\eta\in K$ with reasonably small norm, 
say $\left\|\eta\right\|_H\leq \delta$.
In particular, we can choose $\delta$ such that 
\begin{equation}
\left| F([\theta^*+\eta]) - F([\theta^*]) - 
F'([\theta^*])\eta \right|
\leq \frac{\epsilon }{4} \left\|\eta\right\|_{H}. 
\end{equation} 
Next, choose $\tau$ small enough so 
that $(\theta^*+\tau \theta_k)\in K$ 
and  $\left\|\tau \theta_k\right\|\leq \delta$ 
for all $k\in \mathbb{N}$. For such functions we have
\begin{align}
\left| F([\theta^*+\tau \theta_k])-F([\theta^*+\tau \theta_j]) \right|\geq& 
\left| \tau F'([\theta^*])(\theta_k-\theta_j)\right| - \nonumber \\
& \left| F([\theta^*+\tau \theta_k])-F([\theta^*])- 
\tau F'([\theta^*])\theta_k \right|-\nonumber\\
&\left| F([\theta^*+\tau \theta_j])-F([\theta^*])-
\tau F'([\theta^*])\theta_j\right|\nonumber\\
\geq & \epsilon\tau -\frac{\tau\epsilon}{4}-\frac{\tau\epsilon}{4} \nonumber\\
= & \frac{\epsilon\tau}{2}.
\label{imp}
\end{align}
This means that the functional $F$ is not completely 
continuous. In fact, the inequality \eqref{imp} implies 
that it is not possible to extract a convergent 
sub-sequence from the sequence 
$\{F([\theta^*+\tau \theta_k])\}$, with $\theta_k$ 
bounded. This proves the Lemma. 

\end{proof}

\begin{lemma}
\label{COMPFRE1}
{\bf (Representation of the Fr\'echet derivative)} 
Let $K$ be a compact subset of a real separable 
Hilbert space $H$, and let $F([\theta])$ be a 
continuous real- or complex-valued functional on $K$.
If the Fr\'echet derivative of $F([\theta])$ 
exists at $\theta^*\in K$ then $F'([\theta^*])$ 
admits the unique integral representation
\begin{equation}
F'([\theta^*])\eta = \left(\frac{\delta F([\theta^*])}
{\delta \theta(x)},\eta(x)\right)_H\qquad \forall \eta\in H,
\label{Riesz}
\end{equation}
where $\delta F([\theta^*])/\delta \theta(x)\in H$ is 
the first-order functional derivative \eqref{first_order_FD}. 
\end{lemma}

\begin{proof}
We have seen in Lemma \ref{COMPFRE} that the 
Fr\'echet derivative of a continuous nonlinear functional defined 
on a compact subset of a real separable Hilbert space is a 
compact linear operator. Hence, $\ell([\eta])=F'([\theta^*])\eta$ 
is a bounded linear functional in $H$. By applying the Riesz 
representation theorem \eqref{Riesz} to $\ell$, 
we conclude that there exists a unique 
element $\delta F([\theta^*])/\delta \theta(x)\in H$ 
such that \eqref{Riesz} holds. This proves the Lemma.

\end{proof}

\vs
\noindent 
{\em Example 1:} Let $H=L^2_p([0,2\pi])$ and 
$K$ be the Sobolev sphere \eqref{Sobolevsphere0}, 
which includes its closure in {$L^2_p([0,2\pi])$}.
Consider the nonlinear functional 
\begin{equation}
F([\theta]) = \int_0^{2\pi} \sin(x) \sin^2(\theta(x))dx, \qquad 
\theta\in K.
\label{fun0}
\end{equation}
The Fr\'echet derivative of \eqref{fun0} is
\begin{equation}
F'([\theta]) \eta=
\int_{0}^{2\pi} \sin(x) \sin(2\theta(x))\eta(x)dx,
\qquad \eta \in H.
\end{equation}
For fixed $\theta\in K$, the linear 
functional $F'([\theta]) \eta$ 
is bounded in {$L^2_p([0,2\pi])$}. In fact, 
by the Cauchy-Schwarz inequality we have
\begin{align}
\left|F'([\theta]) \eta\right| \leq &
\underbrace{\left\|\sin(x) \sin(2\theta(x))
\right\|_{L^2_p([0,2\pi])}}_{\left\|F'([\theta])\right\|}
\left\|\eta\right\|_{L^2_p([0,2\pi])}, \qquad 
\left\|F'([\theta])\right\|\leq \sqrt{\pi},
\end{align} 
for all $\eta\in H$. Therefore, for each $\theta\in K$ 
Lemma \ref{COMPFRE} holds, i.e., there exists a 
unique first-order functional derivative
\begin{equation}
\frac{\delta F([\theta])}{\delta \theta(x)}=\sin(x) 
\sin(2\theta(x)),
\label{fd4}
\end{equation}
which is an element of $L^2_p([0,2\pi])$ 
(as a function of $x$).

\section{Cylindrical approximation of nonlinear functionals in real separable Hilbert spaces}
\label{sec:finiteDimapprox}

Let $H$ be a real separable Hilbert space with 
inner product $(\cdot,\cdot)_H$. Any element 
$\theta\in H$ can be represented uniquely in terms 
of an orthonormal basis $\{\varphi_1,\varphi_2,\ldots\}$ as 
\begin{equation}
\theta(x)=\sum_{k=1}^\infty a_k\varphi_k(x),\qquad 
a_k=(\theta,\varphi_k)_H, 
\label{Theta}
\end{equation} 
where the series converges in the norm $\|\cdot\|_H$ induced 
by the inner product $(\cdot,\cdot)_H$.
We introduce the projection operator $P_m$, which 
truncates the series expansion \eqref{Theta} to $m$ terms 
\begin{equation}
\label{proj_op}
P_m \theta = \sum_{k=1}^m (\theta,\varphi_k)_H\varphi_k.
\end{equation} 
Clearly, $P_m$ is an operator from $H$ into the 
finite-dimensional space 
\begin{equation}
D_m = \textrm{span}\{\varphi_1,...,\varphi_m\}.
\label{Dspan}
\end{equation}
With this notation we can represent any nonlinear 
functional $F([\theta])$ in $H$ as a function depending 
on an infinite (countable) number of variables. To this 
end, we substitute \eqref{Theta} into $F([\theta])$ 
to obtain
\begin{equation}
f(a_1,a_2,\ldots ) = F\left(\left[\sum_{k=1}^\infty a_k\varphi_k(x)\right]\right), \qquad {a_k=(\theta,\varphi_k)_H, \qquad \theta\in H.}
\label{functionalFD1}
\end{equation}
A simple way to approximate this functional is to 
restrict its domain to {the range of the 
projection $P_m$, which is the finite-dimensional 
space $D_m$ in \eqref{Dspan}}. 
This reduces $F([\theta])$ to a multivariate 
function $f(a_1,\ldots,a_m)$, which depends on 
as many variables as the number of basis elements 
of $D_m$. Specifially, we have
\begin{equation}
f(a_1,\ldots,a_m)=F\left(\left[\sum_{k=1}^m a_k \varphi_k \right]\right), \qquad {a_k=(\theta,\varphi_k)_H, \qquad \theta\in H.}
\label{functionalFD}
\end{equation}
In the theory of stochastic processes, the set 
\begin{equation}
 \left\{\theta\in H:\, ((\theta,\varphi_1)_H,...,(\theta,\varphi_m)_H)\in B\right\},
\end{equation}
{where $B$ is a Borel set of $\mathbb{R}^m$}, is known as {cylindrical set} (see \cite[p. 55]{Neerven} 
or \cite[p. 45]{Skorokhod}). Therefore, functionals 
of the form \eqref{functionalFD}, i.e., 
\begin{equation} 
f\left( (\theta,\varphi_1)_H, (\theta,\varphi_2)_H, \ldots, 
 (\theta,\varphi_m)_H \right)=F([P_m\theta]), \qquad \theta\in H, 
 \label{cyl}
\end{equation}
{where $f$ is a multivariate function,
are often referred to as cylindrical (or cylinder) 
functionals\footnote{{
In \cite[p. 336]{Friedrichs} and \cite[Ch. 1]{Friedrichs_Book} 
Friedrichs and Shapiro defined {cylinder functionals} 
in a real separable Hilbert space $H$ as those functionals 
which depend on their argument $\theta\in H$ only in as much as 
they depend on $P_m \theta$, where $P_m$ is an orthogonal 
projection on $H$. In other words, 
a cylinder functional on $H$ is a functional 
$F$ such that $F([\theta])=F([P_m\theta])$ for 
all $\theta\in H$. Clearly, \eqref{cyl} is a cylinder 
functional.}}
\cite{Baez,Friedrichs}.} 
{Such functionals play 
a fundamental role, e.g., in the approximation of functional 
integrals arising in quantum field theory \cite{Combe,Justin} 
(see also \ref{sec:functionalintegral}).
\begin{definition}
{(\bf Cylindrical approximation of nonlinear functionals)} 
Let $H$ be a real 
separable Hilbert space, $P_m$ the projection 
operator \eqref{proj_op}, and $F$ a nonlinear 
functional on $H$. We will call $F([P_m\theta])$ 
cylindrical approximation of $F([\theta])$.   
\end{definition}}

\noindent
We will see in section \ref{sec:F} 
that $F([P_m \theta])$ (Eq. \eqref{cyl})
converges uniformly to $F([\theta])$ 
(Eq. \eqref{functionalFD1}) as $m$ goes to infinity 
in any compact subset of a real separable 
Hilbert space $H$.

Next, we study the representation of the Fr\'echet and 
the first-order functional derivatives. 
If $F([\theta])$ is Fr\'echet differentiable 
at $\theta\in H$ with continuous 
Fr\'echet derivative $F'([\theta])$, then 
$\ell(\eta)=F'([\theta])\eta$ is a bounded linear functional. 
Hence, by Riesz's representation theorem, 
there exists a unique element of $H$, which we 
denoted by $\delta F([\theta])/\delta \theta(x)$, 
such that
\begin{equation}
F'([\theta])\eta=\left(\frac{ \delta F\left(\left[\theta \right]\right)}
{\delta \theta(x)},\eta\right)_{H}.
\label{fff1}
\end{equation}
As we pointed out in Lemma \ref{COMPFRE1}, 
$\delta F([\theta])/\delta \theta(x)$ coincides with the 
first-order functional derivative \eqref{first_order_FD}.
Such derivative is an element of $H$, and therefore 
it can be represented in terms of the 
orthonormal basis $\{\varphi_1,\varphi_2,\ldots\}$ as
\begin{equation}
\frac{\delta F([\theta])}{\delta \theta(x)}=
\sum_{k=1}^\infty 
\left(\frac{\delta F([\theta])}{\delta \theta(x)},
\varphi_k\right)_H \varphi_k(x).
\label{FDD}
\end{equation} 
A differentiation of \eqref{functionalFD1} with respect to $a_k$ 
yields\footnote{Equation \eqref{pj} can be equivalently written 
as
\begin{align}
\frac{\partial f}{\partial a_k} = F'([\theta])\varphi_k, \qquad k=1,2,\ldots.
\end{align}
}
\begin{align}
\frac{\partial f}{\partial a_k} = 
\left( \frac{ \delta F\left(\left[\theta \right]\right)}
{\delta \theta(x)},\varphi_k\right)_{H}, \qquad k=1,2,\ldots
\label{pj}
\end{align}
This means that the partial derivative of $f(a_1,a_2,\ldots)$ 
with respect to $a_k=(\theta,\varphi_k)_H$ is the 
projection of the first-order functional derivative of $F$ 
onto the basis element $\varphi_k$. By substituting 
\eqref{pj} into \eqref{FDD} we obtain 
\begin{align}
\frac{\delta F([\theta])}{\delta \theta(x)}=
\sum_{k=1}^\infty 
\frac{\partial f}{\partial a_k} \varphi_k(x). 
\label{FDD1}
\end{align}
This expression emphasizes that the first-order 
functional derivative \eqref{FDD1} is essentially 
a ``dot product'' between the (infinite-dimensional) 
gradient of $f$ and the (infinite-dimensional) vector 
of basis elements.  
Evaluating \eqref{FDD} on the finite-dimensional 
function space $D_m$ yields the cylindrical 
approximation
\begin{align}
\frac{\delta F([P_m \theta])}{\delta \theta(x)}=
\sum_{k=1}^m \frac{\partial f}{\partial a_k} \varphi_k(x) + 
\sum_{k=m+1}^\infty 
\left( \frac{ \delta F\left(\left[P_m\theta \right]\right)}
{\delta \theta(x)},\varphi_k\right)_{H}
\varphi_k(x).
\label{FDD2}
\end{align}
Here $f$ depends solely on the variables 
$(a_1,\ldots, a_m)$. If 
the functional derivative 
$\delta F([P_m \theta])/\delta \theta(x)$ 
is an element of $ D_m$ (as a function of $x$) then 
the second term at the right hand side of \eqref{FDD2} is 
clearly equal to zero.

\vs
\noindent 
{\em Example 1:} Let $F([\theta])$ be a continuously differentiable functional on a real separable Hilbert 
space $H$,  and let $P_m$  be the projection 
operator \eqref{proj_op}. Then from equations
\eqref{fff1} and \eqref{pj} it follows that 
\begin{equation}
F'([P_m\theta])P_m\theta = \sum_{k=1}^m a_k
\frac{\partial f}{\partial a_k},
\label{F'theta}
\end{equation}
where $f$ is the multivariate function 
defined in \eqref{functionalFD}.

\section{Convergence analysis of cylindrical approximations: continuous nonlinear functionals}
\label{sec:F}

In this section we perform a convergence analysis 
for nonlinear functional approximations of the form 
\eqref{functionalFD} in compact subsets of real 
separable Hilbert spaces. We begin by recalling 
an approximation result first obtained by 
Prenter in \cite{Prenter1970}.

\begin{lemma}
\label{UCONVHS}
{\bf (Uniform convergence of cylindrical functional approximations \cite[Lemma 5.3]{Prenter1970})} 
Let $H$ be a real separable Hilbert space,
$K$ a compact subset of $H$, and $P_m : H \to D_m$ the
projection operator \eqref{proj_op}. If $F$ is a continuous 
functional on $H$, then the sequence $F([P_m\theta])$
converges uniformly to $F([\theta])$ on $K$, i.e., 
for all $\epsilon>0$ there exists 
$m_\epsilon\in\mathbb{N}$ such that 
\begin{equation}
\left| F([\theta])-F([P_m\theta])\right| 
\leq \epsilon, \qquad
\end{equation}
for all $\forall m\geq m_\epsilon$ and for all 
$\theta\in K$
\end{lemma}

\noindent
The compactness hypothesis of the 
subset $K$ in  
Lemma \ref{UCONVHS} can be replaced 
by the weaker assumption that $K\subseteq H$ is bounded 
(e.g., a sphere),  and $F$ is uniformly continuous with 
respect to the so-called $S$-topology (see 
\cite{Bertuzzi} for details).

\subsection{Convergence rate}
Lemma \ref{UCONVHS} establishes uniform convergence 
of the functional $F([P_m\theta])$ to $F([\theta])$ on 
compact subsets of real separable Hilbert spaces. 
{We will now address how fast 
the approximation $F([P_m\theta])$ 
converges to $F([\theta])$. 
We will show that for continuously 
differentiable functionals (functionals 
with continuous Fr\'echet derivative) defined 
on compact, convex subset of real separable 
Hilbert spaces, $F([P_m \theta])$ converges 
to $F([\theta])$ at the same rate at which 
$P_m \theta$ converges to $\theta$ in $H$.}
This results follows from the well-known mean 
value theorem.  
\begin{theorem}
\label{MVT}
{\bf(Mean value theorem)} 
Let $F$ be a real-valued continuously differentiable 
functional on a compact convex subset $K$ of a real separable 
Hilbert space $H$. Then for all $\theta_1, \theta_2\in K$ 
the following estimate holds
\begin{equation}
\left| F([\theta_1])-F([\theta_2])\right|\leq
\sup_{\eta\in K}\left\| F'([\eta])\right\|
\left\|\theta_1-\theta_2\right\|_H,
\end{equation}
where $F'([\eta])$ denotes the first-order 
Fr\'echet derivative of $F$.

\end{theorem}

\noindent
We omit the proof as this is a well-known result. 
We simply recall that since $F'([\theta])$ 
is the Fr\'echet derivative of a continuously 
differentiable functional on a compact metric 
space we have that $F'([\theta])$ is a compact 
linear operator (Theorem \ref{COMPFRE}), and 
therefore it is bounded on $K$. Hence, 
\begin{equation}
\sup_{\eta\in K}\left\| F'([\eta])\right\|=M<\infty.\nonumber
\end{equation}
We also emphasize that it is possible to relax the 
assumptions in Theorem \ref{MVT}. For instance, 
it is possible to drop the requirement that $F$ 
is continuously differentiable and  leverage the 
fact that for each $\epsilon >0$ and any pair of 
points $\theta_1,\theta_2\in K$ there 
exists a point $\theta^* \in K$ in which $F$ is 
Fr\'echet differentiable, and  
\begin{equation}
F([\theta_1])-F([\theta_2])<  
F'([\theta^*])(\theta_1-\theta_2)+\epsilon,
\end{equation}
provided the line $\theta_t= t\theta_1+(1-t)\theta_2$ 
is in $K$ for all $t\in [0,1]$. 
However, for the purpose of the present paper we 
shall simply restrict the class of nonlinear 
functionals we study to continuously differentiable 
nonlinear functionals. This allows us to 
obtain the following convergence rate 
result using the mean value Theorem \ref{MVT}. 
\begin{lemma}
\label{LMVT}
{\bf(Convergence rate of cylindrical functional approximations)} 
Let $F$ be a real-valued, continuously 
differentiable functional on a compact and
convex subset $K$ of a real separable Hilbert space $H$. 
Then for all $\theta\in K$ and for any 
finite-dimensional projection 
$P_m$ of the form \eqref{proj_op} 
we have
\begin{equation}
\left| F([\theta])-F([P_m\theta])\right|\leq 
\sup_{\eta\in K}\left\| F'([\eta])\right\|
\left\|\theta-P_m\theta\right\|_H.
\label{estimateMVT}
\end{equation}
In particular, $F([P_m\theta])$ converges 
to $F([\theta])$ for all $\theta\in K$ at the 
same rate as $P_m \theta$ converges to $\theta$ in $H$.
\end{lemma}

\noindent
The proof follows directly from the mean value Theorem  \ref{MVT} 
by setting $\theta_1=\theta$ and $\theta_2=P_m\theta$.

\vs
\noindent
{\em Example 1:} {Consider the compact 
subset  $K\subseteq L^2_w([-1,1])$ defined in equation \eqref{Sobolevsphere}. 
Then, for any continuously differentiable 
functional $F$ on $L^2_w([-1,1])$ we have
\begin{equation} 
\left| F([\theta])-F([P_m \theta])\right|\leq 
\frac{C}{m^{s}},\qquad \forall \theta\in K,
\end{equation}
where $P_m$ is a projection onto ultra-spherical 
polynomials in $[-1,1]$, and $C$ is the (finite) constant
\begin{equation}
C=C_2\rho^2\sup_{\eta\in K}\left\| F'([\eta])\right\|.
\end{equation}
Here, $C_2$ and $\rho^2$ are defined in \eqref{34}.  
If $\theta$ is infinitely 
differentiable, then $F([P_m \theta])$ converges 
to $F([\theta])$ exponentially fast in $m$.}

\section{Convergence analysis of cylindrical approximations: Fr\'echet and functional derivatives}
\label{sec:FD}
In this section we study convergence of cylindrical  
approximations of $F'([\theta])$ and 
$\delta F([\theta])/\delta \theta(x)$ in a compact 
subset $K$ of a separable real Hilbert space $H$. 
We begin with the following

\begin{theorem}
\label{thm:derivatives}
{\bf (Uniform approximation of first-order Fr\'echet  derivatives)}
Let $H$ be a real separable Hilbert space,
$K$ a compact subset of $H$, and $P_m:H \to D_m$ 
the projection operator \eqref{proj_op}. 
If $F$ is continuously differentiable on $K$ with Fr\'echet 
derivative $F'([\theta])$, then the sequence of operators 
$F'([P_m \theta])$ converges uniformly to 
$F'([\theta])$. In other words, for all $\epsilon>0$ 
there exists 
$m_\epsilon\in \mathbb{N}$ 
such that  
\begin{equation}
\left\|F'([\theta]) - F'([P_m\theta]) \right\|=\sup_{\substack{\eta\in H\\\eta\neq 0}} \frac{\left| 
F'([\theta])\eta -F'([P_m\theta])\eta \right|}{\left\|\eta \right\|_H}<\epsilon,
\end{equation}
for all $m\geq m_\epsilon$, and for all $\theta \in K$.
\end{theorem}

\begin{proof}
Let us define the functional 
$G_\eta([\theta])=F'([\theta])\eta$. 
For each fixed $\eta\in H$ we have that $G_\eta([\theta])$ 
is nonlinear and continuous in $\theta$. Hence, we 
can apply Theorem \ref{UCONVHS} to conclude that 
$G_\eta([P_m \theta])$ converges uniformly to 
$G_\eta([\theta])$, i.e., that for each  
$\epsilon_\eta>0$ and there exists $m_\eta\in\mathbb{N}$ 
such that for all $m \geq m_\eta$
\begin{equation}
\label{ineq}
\left|G_\eta([\theta]) - G_\eta([P_m \theta])\right|<\epsilon_\eta,
\qquad \forall \theta\in K.
\end{equation}
Since $F$ is a continuously differentiable 
functional on a compact metric space, the 
Fr\'echet derivative $F'([\theta])$ is a 
compact linear operator (Theorem \ref{COMPFRE}) 
on $H$ for each $\theta\in K$. 
This means that for each fixed $\theta\in K$ 
the linear functional $G_\eta([\theta]) - G_\eta([P_m \theta])$
is bounded 
\begin{equation}
\left|G_\eta([\theta]) - 
G_\eta([P_m \theta])\right| 
\leq \gamma(m)\left\| \eta\right\|_H.
\label{bd}
\end{equation}
By combining \eqref{ineq} and \eqref{bd} we conclude 
that for each $\epsilon>0$ there exists $m_\epsilon\in 
\mathbb{N}$ such  
for 
\begin{equation}
\left\|F'([\theta]) - F'([P_m \theta]) \right\|=  
\sup_{\substack{{\eta\in H}\\\eta\neq 0}}
\frac{\left|\left( F'([\theta])-F'([P_m \theta])\right)\eta\right|}
{\left\|\eta\right\|_H}< 
\epsilon,
\end{equation}
for all $m\geq m_\epsilon$ and for all 
$\theta \in K$. This proves the theorem.

\end{proof}

\noindent
Next, we study convergence of the first-order functional 
derivative \eqref{first_order_FD}. This is relatively 
straightforward given the convergence result we just 
obtained in Theorem \ref{thm:derivatives}. In fact, the linear 
functional $F'([\theta])\eta$ is bounded for each $\theta$ in 
the compact set $K\subseteq H$ and therefore it 
admits the Riesz integral representation 
\begin{equation}
F'([\theta])\eta = \left(\frac{\delta F([\theta])}{\delta \theta(x)},\eta\right)_H \qquad \theta\in K, \qquad \eta \in H,
\label{Fprime}
\end{equation}
where $(\cdot,\cdot )_H$ is the inner product in $H$.
Uniform convergence of $F'([P_m\theta])$ to $F'([\theta])$ 
for all $\theta$ in the compact set $K\subseteq H$ implies 
that for every $\epsilon>0$ there exists 
$m_\epsilon\in\mathbb{N}$ such that   
\begin{equation}
\left|\left(\frac{\delta F([\theta])}{\delta \theta(x)}-
\frac{\delta F([P_m\theta])}{\delta \theta(x)},\eta\right)_H\right|
< \epsilon{\left\|\eta\right\|_H}, \qquad \forall \theta\in K,\qquad \forall\eta\in H\setminus\{0\},
\label{weak}
\end{equation}
and for all $m\geq m_\epsilon$. 

\begin{lemma}
\label{FunctionalDerivativeApprox}
{\bf (Uniform approximation of first-order functional derivatives)}
Let $H$ be a real separable Hilbert space,
$K$ a compact subset of $H$, $\theta\in K$ 
and $P_m: H \to D_m$ the projection \eqref{proj_op}. 
If $F$ is continuously differentiable on $K$ with Fr\'echet derivative 
$F'([\theta])$, then the sequence 
$\delta F([P_m\theta])/\delta \theta (x)$
converges uniformly to $\delta F([\theta])/\delta \theta (x)$. 
In other words, for all $\epsilon>0$ there exists 
$m_\epsilon \in \mathbb{N}$  
such that  for all $m\geq m_\epsilon$
\begin{equation}
\left\| \frac{\delta F([\theta])}{\delta \theta(x)}-
\frac{\delta F([P_m\theta])}{\delta \theta(x)}\right\|_H <\epsilon, \qquad \forall \theta\in K.
\label{67}
\end{equation}

\end{lemma}

\begin{proof}
Consider the linear functional of $\eta\in H$
\begin{equation}
\left(F([\theta])-
F([P_m \theta])\right)\eta = \left(\frac{\delta F([\theta])}{\delta 
\theta(x)}-\frac{\delta F([P_m\theta])}
{\delta \theta(x)},\eta \right)_H.
\label{ffdelta}
\end{equation}
It is well known that the norm of \eqref{ffdelta} is 
\begin{equation}
M([\theta])=\left\| \frac{\delta F([\theta])}{\delta 
\theta(x)}-\frac{\delta F([P_m\theta])}
{\delta \theta(x)}\right\|_{H}.
\end{equation}
By definition, $M([\theta]$) is the smallest number such that 
\begin{equation}
\left|\left(F([\theta])-F([P_m \theta])\right)\eta\right|\leq 
M([\theta])\left\|\eta\right\|_{H}.
\end{equation}
This observation, together with \eqref{weak} allows us 
to conclude that $M([\theta])<\epsilon$ for all $\theta\in K$. 
This proves the theorem. 

\end{proof}

\subsection{Convergence rate}

Let us assume that $G_\eta([\theta])=F'([\theta])\eta$  
is continuously Fr\'echet differentiable with respect to $\theta$ 
in $H$. Denote by $G'_\eta([\theta])$ the first-order 
Fr\'echet derivative and let $K$ be a compact convex subset 
of $H$. By applying the mean value Theorem 
\ref{MVT} we obtain 
\begin{equation}
\left|G_\eta([\theta])-G_\eta([P_m\theta])\right|\leq 
{\sup_{\zeta\in K}}
\left\|G'_\eta([\zeta])\right\|\left\|\theta-P_m\theta\right\|_H,\qquad 
{\forall \eta\in H, \quad \forall \theta\in K}.
\label{sup}
\end{equation}
The Fr\'echet derivative of $G_\eta([\theta])$ 
can be written as\footnote{Note that $F''$ is a 
{compact symmetric bilinear form from 
$H\times H$} into $\mathbb{R}$ or $\mathbb{C}$.}
\begin{equation}
G'_\eta([\theta]) \xi =F''([\theta])\eta \xi. 
\end{equation}
{
If we divide \eqref{sup} by 
$\left\|\eta\right\|_H$ ($\eta\neq 0$) 
and take the supremum over $\eta\in H$ we obtain}
\begin{equation}
\left\| F'([\theta])-F'([P_m\theta])\right\| \leq 
{\sup_{\zeta\in K}} 
\left\| F'' ([\zeta])\right\|\left\|\theta-P_m\theta\right\|_H,
\qquad {\forall \theta\in K},
\label{FD2}
\end{equation}
{where
\begin{equation}
\left\| F'' ([\zeta])\right\|=\sup_{\substack{
{\eta,\xi\in H}\\\eta,\xi\neq 0}}
\frac{\left|F''([\zeta])\eta\xi\right|}
{\left\|\eta\right\|_H\left\|\xi\right\|_H}.
\end{equation}
The symmetric bilinear form $F'' ([\theta])$ 
is continuous on $H \times H$ and therefore it is 
bounded. Moreover, $\left\|F'' ([\theta])\right\|$ is 
continuous in $\theta$ and attains its minimum and 
maximum values in any compact set $K\subseteq H$. 
By equation \eqref{FD2} 
this implies that the Fr\'echet derivative $F'([P_m\theta])$ 
converges to $F'([\theta])$ in $K$ at the same rate 
as $P_m\theta$ converges to $\theta$ in $H$. We summarize 
these results in the following Lemma.
\begin{lemma}
\label{thm:higherderivatives}
{\bf (Convergence rate of first-order Fr\'echet derivatives)}
Let $H$ be a real separable Hilbert space, 
and let $F([\theta])$ be a nonlinear functional 
with continuous first- and second-order 
Fr\'echet derivatives. 
Then for all $\theta$ in a compact convex 
subset $K$ of $H$, and for any finite-dimensional 
projection $P_m$ of the form \eqref{proj_op} we have
\begin{equation}
\left\| F'([\theta])-F'([P_m\theta])\right\| \leq 
{\sup_{\zeta\in K}}\left\| F'' ([\zeta])\right\|\left\|\theta-P_m\theta\right\|_H.
\label{FfD2}
\end{equation}
In particular, $F'([P_m\theta])$ converges uniformly 
to $F'([\theta])$ in $K$ at the same rate as $P_m \theta$ 
converges to $\theta$ in $H$.
\end{lemma}
Convergence rate results for higher-order Fr\'echet 
derivatives can be obtained 
in a similar way.
}

\section{Approximation of linear functional differential equations}
\label{sec:FDE}

Let $\mathcal{F}(H)$ denote a 
Banach space of nonlinear functionals 
from a real separable Hilbert space 
$H$ into $\mathbb{R}$ or $\mathbb{C}$.
In this section we develop necessary 
and sufficient conditions which guarantee 
that the solution to linear functional 
differential equations (FDEs) of the form 
\begin{equation}
\frac{\partial F([\theta],t)}{\partial t} = 
\mathcal{L}([\theta]) F([\theta],t), \qquad  
F([\theta],0)=F_0([\theta]),
\label{thefde}
\end{equation}
can be approximated  by the solution of 
suitable finite-dimensional linear partial differential 
equations. Equation  \eqref{thefde} is a linear 
abstract evolution equation (Cauchy problem) 
in the Banach space $\mathcal{F}(H)$ 
\cite{Guidetti2004}. The linear operator
$\mathcal{L}([\theta])$ is assumed to be 
in $\mathcal{C}(\mathcal{F})$, which 
is the set of closed, densely 
defined and continuous linear operators 
on $\mathcal{F}(H)$. Note that 
$\mathcal{L}([\theta])$ can be unbounded. 
To construct the approximation scheme for the 
FDE \eqref{thefde}, we consider the following 
cylindrical approximation of the solution functional $F$
\begin{equation}
F_m([\theta],t) = F([P_m \theta],t), \qquad \theta\in H,
\label{Fm}
\end{equation}
where $P_m$ is the projection operator \eqref{proj_op}.
We have seen in section \ref{sec:F} that 
$F_m([\theta],t)= f(a_1,\ldots,a_m,t)$ is a 
multivariate function in the $m$ variables 
$a_k=(\theta,\varphi_k)_H$ ($k=1,\ldots, m$)
which converges uniformly to $F([\theta],t)$ in 
every compact subset of $H$, {for 
any fixed time $t$}. 
From a functional analysis perspective, $F_m$ is 
an element of a Banach space of functionals on $H$, which 
we denote by $\mathcal{F}_m(H)$. 
With this notation, we see that the functional 
approximation \eqref{Fm} is essentially 
induced by the application of a 
continuous linear operator 
$B_m:\mathcal{F}(H)\rightarrow \mathcal{F}_m(H)$ 
defined as 
\begin{equation}
B_m F([\theta],t) = F([P_m\theta],t).
\label{Bm}
\end{equation}
Using the operator $B_m$, we 
perform the following decomposition 
of the right hand side of \eqref{thefde} 
\begin{align}
B_m \left(\mathcal{L}([\theta]) F([\theta],t)\right) 
= \mathcal{L}_m([\theta]) F_m([\theta],t) + 
R_m([\theta],t),
\label{decomposition}
\end{align}
where $\mathcal{L}_m([\theta])$ is a
linear operator acting on the $m$-dimensional 
function $F_m([\theta],t)=f(a_1,\ldots,a_m,t)$, 
and $R_m$ is a functional residual.  
The operator $\mathcal{L}_m([\theta])$ can be 
unbounded.
As an example, let $H=L_p^2([0,2\pi])$ (space of 
square integrable periodic functions 
in $[0,2\pi]$) and consider
\begin{equation}
\mathcal{L}([\theta])F([\theta],t) = \int_{0}^{2\pi}
\theta (x) 
\frac{\partial }{\partial x}
\frac{\delta F([\theta],t)}{\delta\theta(x)}
dx, \qquad \theta\in H.
\label{operator}
\end{equation}
A substitution of \eqref{functionalFD} 
and \eqref{FDD2} into \eqref{operator} 
yields
\begin{equation}
B_m \left(\mathcal{L}([\theta])F([\theta],t)\right) =
 \underbrace{\sum_{k,j=1}^m
a_j\frac{\partial f}{\partial a_k}
\int_{0}^{2\pi}
\frac{\partial \varphi_k}{\partial x}
\varphi_j dx}_{\mathcal{L}_m([\theta]) F_m([\theta],t)}+ 
\underbrace{\sum_{k=m+1}^\infty 
\left(\frac{\delta F([P_m\theta],t)}{\delta 
\theta(x)},\varphi_k\right)_{H}\int_{0}^{2\pi} 
\frac{\partial \varphi_k}
{\partial x} P_m\theta dx}_{R_m([\theta],t)},
\label{resDEF}
\end{equation}
where $a_j=(\theta,\varphi_j)_{H}$ ($j=1,\ldots,m$). 
Note that $\mathcal{L}_m([\theta])$ in \eqref{resDEF}
is a linear first-order partial differential operator 
with non-constant coefficients. 
\begin{definition}
\label{def:compatibility}
{\bf (Consistency)} A sequence of linear 
operators $\{\mathcal{L}_m\}\in \mathcal{C}(\mathcal{F}_m)$, 
is said to be consistent (or compatible) 
with a linear operator $\mathcal{L}\in
\mathcal{C}(\mathcal{F})$ if for 
every $F\in D(\mathcal{L})$\footnote{In 
Definition \ref{def:compatibility},
$D(\mathcal{L})$ denotes the domain of 
the operator $\mathcal{L}$.} there exists 
a sequence $F_m\in D(\mathcal{L}_m)$ 
such that 
\begin{equation} 
\left\|F-F_m\right\|\rightarrow 0\quad \text{and}
\quad \left\|\mathcal{L}
F-\mathcal{L}_m F_m
\right\| \rightarrow 0
\end{equation}
as $m\rightarrow \infty$. Moreover, if 
$\left\|\mathcal{L}_m F_m - \mathcal{L} F\right\|
=\mathcal{O}(m^{-p})$ then we say that the sequence 
$\{\mathcal{L}_m\}$ is consistent with $\mathcal{L}$ to order $p$.
\end{definition}
\begin{lemma}
\label{thm:consistency}
{\bf (Consistency of cylindrical approximations to FDEs)}
Let $H$ be a real separable Hilbert space. 
Consider a functional $F\in \mathcal{F}(H)$ 
and a densely defined  closed linear operator 
$\mathcal{L}\in \mathcal{C}(\mathcal{F})$. 
If $\mathcal{L}([\theta])F([\theta])$ is 
continuous in $\theta$ then the 
sequence of operators $\{\mathcal{L}_m\}$ 
defined in \eqref{decomposition} is 
consistent with $\mathcal{L}$ on every 
compact subset $K$ of $H$, provided $
\left\|R_m([\theta])\right\|\rightarrow 0$
as $m\rightarrow \infty$ for all $\theta\in K$.

\end{lemma}
\begin{proof}
Equation \eqref{decomposition} implies 
that
\begin{align}
\left|\mathcal{L}([\theta])F([\theta]) - 
\mathcal{L}_m([\theta])F_m([\theta])\right| 
=& \left|\mathcal{L}([\theta])F([\theta]) - 
B_m\left( \mathcal{L}([\theta])F([\theta])\right)
+R_m([\theta]) \right|.
\label{87}
\end{align}
Since the functional 
$\mathcal{L}([\theta])F([\theta]) $ is  
continuous in $\theta$, we can now 
use the uniform approximation Theorem 
\ref{UCONVHS} and claim that for 
any $\epsilon>0$ there exists 
$m_\epsilon\in \mathbb{N}$ such 
that 
\begin{equation}
\left|\mathcal{L}([\theta])F([\theta]) - 
B_m\left( \mathcal{L}([\theta])F([\theta])\right)
\right|\leq \epsilon, \qquad \forall m\geq m_\epsilon,\quad \forall \theta \in K,
\label{88}
\end{equation}
where $K$ is a compact subset of $H$.
A substitution of \eqref{88} into \eqref{87} yields,
\begin{align}
\left|\mathcal{L}([\theta])F([\theta]) - 
\mathcal{L}_m([\theta])F_m([\theta])\right|  
\leq &\epsilon
+\left|R_m([\theta]) \right|, \qquad \theta \in K.
\end{align}
Hence if $\left|R_m([\theta]) \right|\rightarrow 0$ 
for all $\theta\in K$ as $m\rightarrow\infty$ then 
$\mathcal{L}([\theta])F([\theta]) \rightarrow \mathcal{L}_m([\theta])F_m([\theta])$ for all $\theta$ in $K$. 
By Theorem \ref{UCONVHS} we also have that 
$F_m\rightarrow F$ on $K$. Hence the sequence 
$\{\mathcal{L}_m\}$ is a consistent approximation 
of $\mathcal{L}$.
 
\end{proof}

\begin{corollary}
\label{thm:consistencyORDER}
Under the same assumptions of Lemma
\ref{thm:consistency} if, in addition, 
$K$ is convex, $\mathcal{L}([\theta])F([\theta])$ is 
continuously Fr\'echet differentiable in 
$K$, and $\left\| R_m([\theta])\right\|=\mathcal{O}(
\left\|\theta-P_m\theta\right\|_H)$ 
then $\{\mathcal{L}_m\}$ is consistent 
with $\mathcal{L}$ to the same order 
as $P_m\theta$ converges to $\theta$ 
in $H$. 
\end{corollary}
\begin{proof}
By using the mean value Theorem \ref{MVT}
and equation \eqref{87} we immediately conclude that 
\begin{equation}
\left\|\mathcal{L}([\theta])F([\theta]) - 
\mathcal{L}_m([\theta])F_m([\theta])\right\| 
= \mathcal{O}(\left\|\theta-P_m\theta\right\|_H).
\end{equation}
Hence, $\mathcal{L}_m$ is consistent 
with $\mathcal{L}$ to the same order as 
$P_m\theta$ converges to $\theta$ in $K$.

\end{proof}

\vs
\noindent
{\em Example 1:} 
Let $H=L_p^2([0,2\pi])$ be the space of 
square integrable periodic functions in $[0,2\pi]$, 
$\{\varphi_k\}$ an orthonormal Fourier 
basis in $H$, and $K\subseteq H$ 
the Sobolev sphere \eqref{Sobolevsphere0} (together 
with its closure in $H$). 
We have seen in section \ref{sec:functionalonmetric} 
that $K$ is a compact subset of $H$. 
We now show that the operator $\mathcal{L}_m([\theta])$ 
defined in \eqref{resDEF} is a consistent approximation 
of the operator \eqref{operator}, in the compact set $K$. 
For all $m$ larger than some fixed $m_0\in \mathbb{N}$ 
we have
\begin{align}
\left|R_m([\theta],t)\right|=&\left|\sum_{k=m+1}^\infty \left(\frac{\delta F([P_m\theta],t)}{\delta 
\theta(x)},\varphi_k\right)_{H}\int_{0}^{2\pi} \frac{\partial \varphi_k}
{\partial x} P_m\theta dx \right|\nonumber\\
\leq& 
\sum_{k=m+1}^\infty\left|\left(\frac{\delta F([P_m\theta],t)}{\delta 
\theta(x)},\varphi_k\right)_{H}\right|\left\| \varphi_k
\right\|_{H}
\left\|\frac{\partial  (P_m \theta)}{\partial x}\right\|_{H}\nonumber\\
\leq& \gamma
\sum_{k=m+1}^\infty\left|\left(\frac{\delta F([P_m\theta],t)}{\delta 
\theta(x)},\varphi_k\right)_{H}\right|,
\label{res1}
\end{align}
{where $\gamma$ is a constant} 
independent of $m$. 
To obtain the last inequality, we 
used the fact that $\varphi_k$ is orthonormal 
in $H$ ($\left\| \varphi_k\right\|_{H}=1$), and 
that $\left\|\partial(P_m \theta)/\partial x\right\|$ 
converges to $\left\|\partial \theta/\partial x\right\|$ 
in $H$ (uniformly in $\theta\in K$). The proof of this 
statement is based on the following inequalities 
\cite[p. 38]{spectral} 
\begin{equation}
\left|\left\|\frac{\partial  \theta}{\partial x}\right\|_H - 
\left\|\frac{\partial  (P_m \theta)}{\partial x}\right\|_H\right|
\leq\left\|\frac{\partial  \theta}{\partial x}-\frac{\partial  (P_m \theta)}{\partial x}\right\|_H \leq \frac{C}{m^{s-1}}
\left\|\frac{d^s\theta}{dx^s}\right\|_H\leq \frac{\rho C}{m^{s-1}}.
\label{spectralDer}
\end{equation}
In the last inequality, we used the fact 
that $\theta$ is in the Sobolev sphere 
\eqref{Sobolevsphere0}. From \eqref{spectralDer} it 
follows that 
\begin{equation}
\left\|\frac{\partial  (P_m \theta)}{\partial x}\right\|_H\leq \left\|\frac{\partial   \theta}{\partial x}\right\|_H+\frac{\rho C}{m^{s-1}}\leq \kappa \rho + \frac{\rho C}{m^{s-1}}
\leq \underbrace{\kappa \rho + \frac{\rho C}{m_0^{s-1}}}_{\gamma}, \qquad \forall m\geq m_0,
\label{84}
\end{equation} 
where we repeatedly applied the Poincar\'e inequality 
$\left\|f\right\|_{H}\leq g \left\|\partial f/\partial x\right\|_{H}$ to obtain the constant $\kappa$.
Equation \eqref{84} defines the constant 
$\gamma$ appearing in \eqref{res1}.
At this point we recall that the functional 
derivative $\delta F([P_m\theta],t)/\delta \theta(x)$ 
converges strongly in $H$ to 
$\delta F([\theta],t)/\delta \theta(x)$
as $m$ goes to infinity for all $\theta\in K$
(Theorem \ref{FunctionalDerivativeApprox}).
This implies that \eqref{res1} goes to zero 
for all $\theta\in K$ as $m$ goes 
to infinity\footnote{Recall that if 
$f_n\rightarrow f$ is a strongly convergent 
sequence in a Hilbert space $H$ then, for any given 
$\epsilon > 0$ and any orthonormal basis of 
$H$ there exists $m_\epsilon\in \mathbb{N}$ 
such that
\begin{equation}
\sum_{k=m_\epsilon+1}^\infty 
\left|(f_n,\varphi_k)_H\right|< \epsilon.
\end{equation}}, i.e., 
\begin{equation}
\max_{\theta\in K}\left| R_m([\theta],t)\right|\rightarrow 0,
\qquad \forall t\in [0,T].
\label{consistency}
\end{equation}
The rate of convergence depends 
on the regularity of the first-order 
functional derivative 
$\delta F([\theta],t)/\delta \theta(x)$ 
as a function of $x$. In particular, if 
$\delta F([\theta],t)/\delta \theta(x)$ is
infinitely differentiable in $x$, 
then  \eqref{consistency} 
goes to zero exponentially fast 
with $m$ \cite[p. 36]{spectral}.

\subsection{Cylindrical approximations to FDEs: stability and convergence}
Let us now consider the $m$-dimensional linear PDE
\begin{equation}
\frac{\partial F_m([\theta],t)}{\partial t}= \mathcal{L}_m([\theta]) F_m([\theta],t),  \qquad
F_m([\theta],0)= B_m F_0([\theta]),
\label{thepde}
\end{equation}
where $B_m$ and $F_0$ are defined in \eqref{Bm}
and \eqref{thefde}, respectively. If the conditions 
of Lemma \ref{thm:consistency} are satisfied 
then we  say that the PDE \eqref{thepde} is a 
consistent approximation of the FDE \eqref{thefde}.
Moreover if $\{\mathcal{L}_m\}$ in \eqref{thepde} 
is consistent with $\mathcal{L}$ to order $p$ then we 
say that the PDE \eqref{thepde} is consistent with 
the FDE \eqref{thefde} with order $p$.

A fundamental question at this point is 
whether the solution of \eqref{thepde} 
converges to the solution of the FDE \eqref{thefde} as 
we send $m$ to infinity. 
The Trotter-Kato approximation theorem for
abstract evolution equations in Banach spaces
\cite[p. 209]{engel1999one} states that this 
is indeed the case, provided the initial 
value problem \eqref{thepde} is 
``stable'' in the following sense. 
\begin{definition}
\label{def:stability}
{\bf (Stability)} 
Suppose that the linear operator 
$\mathcal{L}_m$ in \eqref{thepde} 
generates a strongly continuous semigroup 
$e^{t \mathcal{L}_m}$. We say that 
the FDE approximation \eqref{thepde} is 
stable if there are two constants $M$ and 
$\omega$ independent of $m$ such that  
$\left\| e^{t \mathcal{L}_m}\right\|\leq Me^{\omega t}$.
\end{definition} 
We now have all elements to state 
a version of the Trotter-Kato  
theorem \cite[p. 8]{Guidetti2004} 
that holds for cylindrical approximations 
of functional differential equations. 

\begin{theorem}
\label{thm:FDEconvergence}
{\bf (Convergence of cylindrical 
approximations to FDEs)}
Suppose that the initial value problem \eqref{thefde} is 
well-posed in the time interval $[0,T]$ ($T$ finite), 
and that $\mathcal{L}([\theta])\in 
\mathcal{C}(\mathcal{F})$ generates a 
strongly continuous semigroup in $[0,T]$.
Then the FDE approximation \eqref{thepde} is 
stable and consistent in a compact subset $K$ 
of a real separable Hilbert space $H$ 
if and only if it is convergent, i.e.,
\begin{equation}
\max_{t\in [0,T]}\max_{\theta\in K}
\left| F_m([\theta],t) - F([\theta],t) \right|
\rightarrow 0 
\label{convergence}
\end{equation}
as $m\rightarrow \infty$, provided $F_m([\theta],0)\rightarrow F_0([\theta])$. 
\end{theorem}
The proof of this theorem can be found 
in \cite[p. 210]{engel1999one}. In summary, to 
prove that a cylindrical approximations 
to FDEs is convergent we can proceed as follows:

\begin{enumerate}
\item [a)]  Construct the multivariate PDE 
\eqref{thepde} and show that such PDE
is a consistent approximation to the FDE 
\eqref{thefde} (Lemma \ref{thm:consistency}) 

\item [b)] Study stability of \eqref{thepde}.  
This is a PDE-specific result stating 
that it is possible to control some norm of 
the solution of \eqref{thepde} by a constant 
multiple of a suitable norm of the initial condition, 
and all the norms involved (including the constant)
do not depend on $m$. The simplest 
stability results arise from energy inequalities, 
e.g., for PDEs with continuous and coercive linear 
operators $\mathcal{L}_m$.

\item [c)] Apply Theorem \ref{thm:FDEconvergence} 
to claim that if a) and b) are satisfied, then 
the solution of the PDE \eqref{thepde} 
converges uniformly to the solution 
of the FDE \eqref{thefde} as the 
number of independent variables $m$ goes to 
infinity.
\end{enumerate}

\vs
\noindent 
{\em Example 2:} Consider the 
initial value problem
\begin{equation}
\frac{\partial F([\theta],t)}{\partial t}=
\int_{0}^{2\pi}
\theta (x) 
\frac{\partial }{\partial x}
\frac{\delta F([\theta])}{\delta\theta(x)}dx,
\qquad F([\theta],0)=F_0([\theta]).
\label{example2}
\end{equation}
{The FDE \eqref{example2} 
is the Hopf equation corresponding to the 
linear PDE  
\begin{equation}
\frac{\partial u}{\partial t}=
\frac{\partial u}{\partial x}, \qquad
u(x,0)=u_0(x),
\label{PDElin}
\end{equation}
where $u_0(x)$ is random and 
periodic in $[0,2\pi]$. 
To show this, let 
\begin{equation}
F([\theta],t)= \mathbb{E}\left\{\exp
\left(i\int_{0}^{2\pi}u(x,t)\theta(x)dx\right)\right\}
\label{HPDE}
\end{equation}
be the Hopf functional associated 
with the solution to \eqref{PDElin}.
The expectation operator $\mathbb{E}\left\{\cdot\right\}$ 
in \eqref{HPDE} is an integral over the probability 
measure of $u_0(x)$. Differentiation of \eqref{HPDE}
with respect to time yields
\begin{align}
\frac{\partial F([\theta],t)}{\partial t}=&
i\mathbb{E}\left\{\exp\left(i\int_{0}^{2\pi} u(x,t)\theta(x)dx\right)
i\int_{0}^{2\pi} \frac{\partial u(x,t)}{\partial t}\theta(x)dx\right\}\nonumber\\
 =&
i\mathbb{E}\left\{\exp\left(i\int_{0}^{2\pi} u(x,t)\theta(x)dx\right)
i\int_{0}^{2\pi} \frac{\partial u(x,t)}{\partial x}\theta(x)dx\right\}
\nonumber\\
 =&
\int_{0}^{2\pi} \frac{\partial }{\partial x}
\left(\frac{\delta F([\theta],t)}{\delta \theta(x)}\right)
\theta(x)dx.
\end{align}}

\noindent
We assume that $\theta$ is in the compact 
set $K\subseteq L^2_p([0,2\pi])$ defined 
in \eqref{Sobolevsphere0}. This is  
domain in which we solve the FDE \eqref{example2}. 
Let $\{\varphi_1,\varphi_2,\ldots\}$ 
be an orthonormal basis of $L^2_p([0,2\pi])$. 
By equations \eqref{resDEF} and \eqref{consistency}  
we have that the $m$-dimensional PDE
\begin{equation}
\frac{\partial f}{\partial t}=\sum_{k,j=1}^m
\frac{\partial f}{\partial a_k}\int_{0}^{2\pi}
\frac{\partial \varphi_k}{\partial x}\varphi_jdx, 
\qquad f(a_1,\ldots,a_m,0) = F_0([P_m \theta]), 
\label{example2C}
\end{equation}
where $a_k=(\theta,\varphi_k)_{L^2_p([0,2\pi])}$, is 
a consistent cylindrical approximation to 
the FDE \eqref{example2}.
Next, we show that such approximation is stable 
in the sense of Definition \ref{def:stability}.
By using the method of characteristics \cite{Rhee}
it is straightforward to to show that the solution of  \eqref{example2C} can be bounded as
\begin{equation}
|f(a_1,\ldots, a_m,t)| \leq  
\left\|f_0\right\|_{L^\infty(\mathbb{R}^m)},
\label{stability}
\end{equation}
where $f_0 = f(a_1,\ldots, a_m,0)$.
Hence, if the $L^\infty$ norm of $f_0$ is bounded by 
a constant $\kappa$ that is independent of $m$, 
then \eqref{example2C} is stable in the 
$L^\infty(\mathbb{R}^m)$ norm\footnote{An 
example of a cylindrical functional that is bounded 
in the $L^{\infty}(\mathbb{R}^m)$ norm is 
\begin{equation}
f(a_1,\ldots,a_m,0)=\int_0^{2\pi} \sin(x)\sin
\left(\sum_{k=1}^m a_k \varphi_k(x)\right)dx,\qquad a_k=(\theta,\varphi_k)_{L^2_p([0,2\pi])}.
\end{equation}
In section \ref{sec:convFDE} we show that 
the solution of \eqref{example2C} corresponding 
to such initial condition converges uniformly in $K$ 
and exponentially fast in $m$ to the 
solution of \eqref{example2}.}. 
Such strong bound implies that the 
solution \eqref{example2C} is also 
bounded in the $L^2_\mu$ norm, where $\mu$ is 
the measure defined in \eqref{functionalmeasure}. In fact,
we have 
\begin{equation}
\left\|f\right\|_{L^2_\mu}\leq \left\|f_0\right\|_{L^\infty}\leq \kappa,\qquad 
\forall m\in\mathbb{N}.
\label{stability1}
\end{equation}
Note that this also implies that the functional 
integral defined in \eqref{fInt}-\eqref{fdint} 
converges, as it is bounded by the same constant 
$\kappa$ independently of $m$.
By using Theorem \ref{thm:FDEconvergence} 
we conclude that the solution of 
the PDE \eqref{example2C} converges uniformly 
to the solution of the FDE \eqref{example2} in $K$, 
as the number of independent variables $m$
goes to infinity.

\vs\noindent
We now have the main tools to study 
convergence of cylindrical approximations to FDEs. 
The main result is Theorem \ref{thm:FDEconvergence} 
which is based on the Trotter-Kato approximation 
theorem for abstract evolution equations in Banach spaces
\cite{Trotter1958,Kato1959}. The theorem states that 
stable consistent approximations of FDEs are convergent, 
but it does not provide an estimate on the rate of 
convergence of the approximation, i.e., 
how fast $F_m$ converges to $F$. Estimates 
of such rate of convergence are available in rather 
general cases (e.g., \cite{Campiti2008}), 
but a thorough analysis for cylindrical 
approximations of FDEs is lacking.
Nevertheless, in section 
\ref{sec:FDEconvergenceNumerics} we will 
show that the convergence rate of the 
cylindrical approximation to a prototype 
FDE can be exponential.

{
\section{Approximation of nonlinear functionals and 
FDEs in real Banach spaces admitting a basis}
\label{sec:Banach}
In this section we outline the extension 
of the functional approximation theory 
we developed in real separable Hilbert spaces to 
nonlinear functionals and FDEs defined on
compact subsets of real Banach spaces 
admitting a basis\footnote{{A 
real Banach space with a basis $\{\varphi_k\}$ 
is necessarily separable since the set 
of all finite linear combinations 
$\sum_{k}a_k\varphi_k$ forms 
a countable dense subset of $X$ \cite{Morrison2001}. 
The longstanding question of whether every 
separable Banach space possesses a basis was 
answered by Per Enflo \cite{Enflo} in 1973. 
He showed that there do exist separable 
Banach spaces that do not possess 
a basis.}}.
Well-known examples of such Banach spaces 
are:
\begin{itemize}
\item $C^{(0)}([0,1])$ (space of continuous 
functions in $[0,1]$) \cite[\S 5.2]{Morrison2001};
\item $L^p(\Omega)$ for $1<p<\infty$ (Lebesgue space 
defined on a compact domain $\Omega\subseteq\mathbb{R}^n$)
\cite[Theorem 2.1]{Bellout};
\item $W^{k,p}(\Omega)$ for $1<p<\infty$ (Sobolev 
space defined on a compact domain 
$\Omega\subseteq\mathbb{R}^n$ with smooth or 
Lipschitz boundary \cite{Figiel,Matveev2002}).
\end{itemize}
Before we present the main results, 
let us briefly recall the definition and 
the basic properties of Schauder bases 
in Banach spaces.}
{
\begin{definition}
\label{def:Schauderbasis}
{\bf (Schauder basis)}
A Schauder basis of a Banach space $X$ is sequence 
of linearly independent elements $\varphi_k\in X$ such 
that every $\theta\in X$ can be uniquely represented 
as 
\begin{equation}
\theta =\sum_{k=1}^\infty a_k([\theta]) \varphi_k,
\label{schouderbasis}
\end{equation}
where $a_k: X \rightarrow \mathbb{R}$ is 
a sequence of bounded linear 
functionals\footnote{{It is shown 
in \cite[p. 20]{Singer1} that 
$1 \leq |a_k|\left\|\varphi_k\right\|_X\leq 2C$ 
for all $k\in \mathbb{N}$, where 
$C\geq 1$ is the so-called basis constant. In the 
case of real separable Hilbert spaces the 
linear functionals $a_k([\theta])$ are given by
by $a_k([\theta])=(\theta,\varphi_k)_H$ (see Eq. \eqref{Theta}), 
and they are obviously bounded.}} 
uniquely determined by the basis $\{\varphi_k\}$. 
\end{definition}
As is well-known, every basis in a Banach 
space is a Schauder basis (see, e.g., 
\cite[p. 20]{Singer1} or 
\cite[Proposition 5.3]{Morrison2001}). 
Hence, hereafter we will drop the adjective 
``Schauder'' when referring to a basis 
in Banach space.
\subsection{Compact subsets of real separable Banach spaces}
\label{sec:compactBanach}
Just like in the case of functional 
approximation in real separable Hilbert spaces, 
all approximation results we present hereafter 
hold in compact subsets of Banach spaces with 
a basis. Characterizing such compact 
subsets, is not as straightforward 
as in the case of Hilbert spaces (see the 
introduction of section \ref{sec:functionalonmetric} 
and Theorem \ref{compactSets}).
Nevertheless, compactness results are 
available in rather general cases. 
For instance, the Arzel\`a-Ascoli
theorem \cite{Rudin} provides necessary and sufficient 
conditions for a set $K\subseteq C^{(0)}(\Omega)$ 
($\Omega$ compact subset of $\mathbb{R}^n$) 
to be pre-compact. Specifically, the theorem states that 
$K$ is pre-compact in the topology induced by 
the uniform norm if and only if $K$ is 
equicontinuous and pointwise bounded. 
By using the Arzel\`a-Ascoli  theorem it is 
straightforward to prove, e.g., that the set of 
Lipschitz-continuous (with the same Lipschitz 
constant) probability density functions on $\Omega$  
is pre-compact in $C^{(0)}(\Omega)$. 
%
%
A similar compactness result, known 
as Kolmogorov-Riesz theorem 
\cite{Hanche-Olsen,Hanche-Olsen2}, can be obtained 
in  $L^p(\mathbb{R}^n)$ and $W^{1,p}(\mathbb{R}^n)$ 
($1\leq p<\infty$). Such theorem can be stated 
as follows. 
\begin{theorem}
\label{lemma:subsetsLp}
{\bf (Compact subsets of $L^p$ \cite{Hanche-Olsen2})}
A subset $K$ of $L^p(\mathbb{R}^n)$  ($1\leq p<\infty$) 
is pre-compact if and only if 
\begin{equation}
\lim_{|h|\rightarrow 0}
\left\|f(x+h)-f(x)\right\|_{L^{p}(\mathbb{R}^n)}=0\qquad \text{and}\qquad 
\lim_{r\rightarrow \infty}
\int_{|x|\geq r} |f(x)|^pdx = 0,
\label{condKK}
\end{equation}
for all $f\in K$.
\end{theorem}
The two conditions in \eqref{condKK} are known as 
equicontinuity and equitight conditions.  
Theorem \ref{lemma:subsetsLp} also holds 
in $L^p(\Omega)$, where $\Omega$ is a compact 
subset of $\mathbb{R}^n$.
More generally, one can use well-known 
compact embedding results such as 
the Rellich-Kondrachov theorem \cite[\S 6]{Adams}.
Such theorem states that the Sobolev space 
$W^{k,p}(\Omega)$ defined on 
a compact domain $\Omega\subseteq\mathbb{R}^n$ 
with differentiable boundary is compactly 
embedded in $W^{l,q}(\Omega)$, provided 
$k>l$ and $k-p/n>l-q/n$. This means, for example, 
that a closed sphere in $W^{k,p}(\Omega)$ is pre-compact 
in $L^q(\Omega)$ if $k>(p-q)/n$.}

{
\subsection{Approximation results for nonlinear functionals, functional derivatives and FDEs}
\label{sec:Banach_approx}
Let $X$ be a Banach with basis $\{\varphi_k\}$ and let
$D_m=\text{span}\{\varphi_1,\ldots,\varphi_m\}$.  
Define the linear projection operator
$P_m: X\rightarrow D_m$ 
\begin{equation}
P_m \theta = \sum_{k=1}^m a_k([\theta]) \varphi_k.
\label{PmBanach}
\end{equation}
It is straightforward to show that $P_m$ is bounded 
and that $P_m \theta$ converges uniformly to $\theta$ 
in every compact subset of $X$ as $m$ goes to infinity. 
In fact, we have the following 
\begin{lemma} 
\label{lemma:ApproxProperty}
Let $X$ be a Banach space with 
basis $\{\varphi_k\}$, and let $K$ be compact 
subset of $X$. Then for each $\epsilon>0$ 
there exists $m_\epsilon\in \mathbb{N}$ such that  
\begin{equation}
\left\|\theta - P_m \theta\right\|_X< \epsilon,
\qquad \forall m\geq m_\epsilon,\qquad \forall \theta\in K. 
\label{eq:AP}
\end{equation}
\end{lemma}}
%


\noindent
{The uniform convergence 
result \eqref{eq:AP} is known as 
``approximation property'' in Banach 
space theory \cite{McArthur} 
(see \cite[p. 638]{James} for a proof).
Hence, Lemma \ref{lemma:ApproxProperty} shows 
that every Banach space with a basis has the 
approximation property. We remark that in a real 
separable Hilbert space the uniform approximation property  
follows immediately from the monotonicity 
of the sequence $f_m([\theta])=\left\|\theta- 
P_m\theta\right\|_H$ (Parseval's identity 
implies $f_{m+1}([\theta])\leq f_{m}([\theta])$), 
and Dini's theorem.}
{
\begin{lemma}
\label{UCONVHS_BANACH}
{\bf (Uniform convergence of functional approximations)} 
Let $X$ be a real Banach space with a basis,
$K$ a compact subset of $H$, and $P_m$ the
projection operator \eqref{PmBanach}. 
If $F$ is a continuous functional on $X$, 
then the sequence $F([P_m\theta])$
converges uniformly to $F([\theta])$ on $K$, i.e., 
for all $\epsilon>0$ there exists 
$m_\epsilon\in\mathbb{N}$ such that 
\begin{equation}
\left| F([\theta])-F([P_m\theta])\right| 
\leq \epsilon, \qquad 
\forall m\geq m_\epsilon,\qquad \forall \theta\in K.
\end{equation}
\end{lemma}
\vspace{0.5cm}
\noindent
The proof of this Lemma is essentially the same 
as the proof of Lemma 5.1 in \cite{Prenter1970}. 
We only need to replace the first equation at 
page 380 in \cite{Prenter1970} with \eqref{eq:AP}.}
{As before, we will refer to  
$F([P_m \theta])$ as ``cylindrical approximation''\footnote{{A 
cylinder functional on a real Banach space $X$ admitting 
a basis is a functional $f$ such that 
$F([\theta])=F([P_m\theta])$ for all $\theta\in X$.
This definition relies on the fact that 
$F([P_m \theta])$ is a multivariate function 
of the coefficients $a_k([\theta])$, which 
define the cylinder set \cite[p. 55]{Neerven}
\begin{equation}
\{\theta\in X:\, (a_1([\theta]),\ldots,a_m([\theta]))\in B\},
\end{equation}
where $B$ is a Borel set of $\mathbb{R}^m$.}} 
of $F([\theta])$.}

{
Most of the approximation results we obtained for 
nonlinear functionals, Fr\'echet derivatives, 
functional derivatives and FDEs in compact subsets of 
in real separable Hilbert spaces hold also 
in compact subsets Banach spaces admitting a 
basis. Hereafter we list the most 
important ones. The proofs are the same as in 
the case of real separable Hilbert 
spaces, and therefore omitted.
\begin{lemma}
\label{COMPFRE_Banach}
{\bf (Compactness of first-order Fr\'echet derivatives)} 
Let $K$ be a compact subset of a real 
Banach space $X$ admitting a basis,
and let $F([\theta])$ be a continuous 
real- or complex-valued 
functional on $X$. If the Fr\'echet 
derivative $F'([\theta^*])$ exists 
at $\theta^*\in K$, then it is a compact linear 
operator.
\end{lemma}
\begin{lemma}
\label{LMVT_Banach}
{\bf(Convergence rate of functional approximations)} 
Let $F$ be a real-valued, continuously 
differentiable functional on a compact and
convex subset $K$ of a real Banach space $X$ 
admitting a basis. Then for all $\theta\in K$ and for any 
finite-dimensional projection 
$P_m$ of the form \eqref{PmBanach} 
we have
\begin{equation}
\left| F([\theta])-F([P_m\theta])\right|\leq 
\sup_{\eta\in K}\left\| F'([\eta])\right\|
\left\|\theta-P_m\theta\right\|_X.
\label{estimateMVT_Banach}
\end{equation}
In particular, $F([P_m\theta])$ converges uniformly 
to $F([\theta])$ in $K$ at the same rate as $P_m \theta$ 
converges to $\theta$ in $X$.
\end{lemma}
\noindent 
Convergence rate estimates for 
$\left\|\theta-P_m\theta\right\|_X$ are 
available, e.g., in \cite{DeVore,Matveev2002,Figiel}.
\begin{theorem}
\label{thm:derivatives_Banach}
{\bf (Uniform approximation of first-order Fr\'echet  derivatives)}
Let $X$ be a real Banach space admitting a basis,
$K$ a compact subset of $X$, $\theta\in K$ 
and $P_m$ the projection operator \eqref{PmBanach}. 
If $F$ is continuously differentiable on $K$ with Fr\'echet 
derivative $F'([\theta])$, then the sequence of operators 
$F'([P_m \theta])$ converges uniformly to 
$F'([\theta])$. In other words, for all $\epsilon>0$ 
there exists 
$m_\epsilon\in \mathbb{N}$ 
such that  
\begin{equation}
\left\|F'([\theta]) - F'([P_m\theta]) \right\|=
\sup_{\substack{\eta\in H\\\eta\neq 0}} \frac{\left| 
F'([\theta])\eta -F'([P_m\theta])\eta \right|}
{\left\|\eta \right\|_X}<\epsilon,
\end{equation}
for all $m\geq m_\epsilon$, and for all $\theta \in K$.
\end{theorem}
\begin{lemma}
\label{thm:higherderivatives_Banach}
{\bf (Convergence rate of first-order Fr\'echet derivatives)}
Let $X$ be a real Banach admitting a 
basis, and let $F([\theta])$ be a 
nonlinear functional with continuous 
first- and second-order Fr\'echet derivatives.
Then for all $\theta$ in a convex compact
subset $K$ of $X$, and for any projection 
$P_m$ of the form \eqref{PmBanach} we have
\begin{equation}
\left\| F'([\theta])-F'([P_m\theta])\right\| \leq 
{\sup_{\zeta\in K}}\left\| F'' ([\zeta])\right\|\left\|\theta-P_m\theta\right\|_X.
\label{FfD2_banach}
\end{equation}
In particular, $F'([P_m\theta])$ converges uniformly 
to $F'([\theta])$ in $K$ at the same rate 
as $P_m \theta$ converges to $\theta$.
\end{lemma}
Regarding the extension of the approximation result for 
the first-order functional derivative, i.e., Lemma \ref{FunctionalDerivativeApprox}, we can leverage various
generalizations of the Riesz representation theorem 
\eqref{first_order_L} to specific Banach spaces. 
Hereafter we consider the $L^p$ generalization. 
More general versions may involve measure theory, e.g., 
in the case of spaces of continuous functions defined 
on compact subsets of $\mathbb{R}^n$ (see Eq. \eqref{Haus}).
\begin{lemma}
\label{FunctionalDerivativeApprox_BANACH}
{\bf (Uniform approximation of first-order functional derivatives)}
Let $\Omega$ be a compact subset of $\mathbb{R}^n$,  
$K$ a compact subset of $L^p(\Omega)$ ($1< p<\infty$),
and $P_m$ the projection operator \eqref{PmBanach}. 
If $F$ is continuously differentiable on $K$ 
with Fr\'echet derivative $F'([\theta])$, then the 
sequence $\delta F([P_m\theta])/\delta \theta (x)$
converges uniformly to 
$\delta F([\theta])/\delta \theta (x)$
in $L^q(\Omega)$, where $1/p +1/q=1$.
In other words, for all $\epsilon>0$ there 
exists $m_\epsilon \in \mathbb{N}$  
such that  for all $m\geq m_\epsilon$
\begin{equation}
\left\| \frac{\delta F([\theta])}{\delta 
\theta(x)}-\frac{\delta F([P_m\theta])}
{\delta \theta(x)}\right\|_{L^q(\Omega)} 
<\epsilon, \qquad \forall \theta\in K.
\label{67_banach}
\end{equation}
\end{lemma}}
\begin{proof}
{By Lemma \ref{COMPFRE_Banach} 
the Fr\'echet derivative $F'(\theta)$ is a 
compact linear operator in $L^p(\Omega)$ 
for each $\theta \in K$. Hence, $F'([\theta])\eta$ is a bounded 
linear functional in $L^{p}(\Omega)$ for each $\theta\in K$. 
By using the the Riesz representation theorem
we conclude that there exists a unique function
$\delta F([\theta])/\delta \theta(x)\in L^q(\Omega)$ 
with $1/q =1- 1/p$ 
such that 
\begin{equation}
F'([\theta])\eta =\int_{\Omega} \frac{\delta F([\theta])}
{\delta \theta(x)}\eta(x)dx.
\end{equation}
By applying Theorem \ref{thm:derivatives_Banach} 
we obtain that 
\begin{equation}
\left|\int_{\Omega} \left(\frac{\delta F([\theta])}
{\delta \theta(x)}-\frac{\delta F([P_m\theta])}
{\delta \theta(x)}\right)\eta(x)dx\right|<\epsilon
\left\|\eta\right\|_{L^p(\Omega)},\qquad \forall 
\eta\in L^p(\Omega)\setminus\{0\}, 
\qquad \forall\theta\in K,
\label{106}
\end{equation}
As is well-known, the norm of the linear functional 
$\left(F([\theta])-
F([P_m \theta])\right)\eta$ (linear functional of 
$\eta\in L^p(\Omega)$) is 
\begin{equation}
M([\theta])=\left\| \frac{\delta F([\theta])}{\delta 
\theta(x)}-\frac{\delta F([P_m\theta])}
{\delta \theta(x)}\right\|_{L^q(\Omega)}.
\end{equation}
By definition $M([\theta]$) is the smallest number such that 
\begin{equation}
\left|\left(F([\theta])-F([P_m \theta])\right)\eta\right|\leq 
M([\theta])\left\|\eta\right\|_{L^p(\Omega)}.
\end{equation}
This fact, together with \eqref{106} allow us to conclude that 
$M([\theta])<\epsilon$ for all $\theta\in K$. This proves the theorem.}

\end{proof}

\noindent
{A few comments on Lemma \ref{FunctionalDerivativeApprox_BANACH} 
are necessary at this point. First, compact 
subsets of $L^p(\Omega)$ are identified by the 
equicontinuity and the equitight conditions in Theorem \ref{lemma:subsetsLp}. Second, we excluded the 
case $p=1$ as the Banach space $L^1(\Omega)$ does not 
admit a basis.}
{Regarding approximation 
of FDEs in real Banach spaces with a basis, 
we have the following results. 
\begin{lemma}
\label{thm:consistency_Banach}
{\bf (Consistency of cylindrical approximations to FDEs)}
Let $X$ be a real Banach space with 
a basis. Consider a functional $F\in \mathcal{F}(X)$ 
and a densely defined  closed linear operator 
$\mathcal{L}\in \mathcal{C}(\mathcal{F})$. 
If $\mathcal{L}([\theta])F([\theta])$ is 
continuous in $\theta$, then the 
sequence of operators $\{\mathcal{L}_m\}$ 
defined in \eqref{decomposition} is 
consistent with $\mathcal{L}$ on every 
compact subset $K$ of $X$, provided $
\left\|R_m([\theta])\right\|\rightarrow 0$
as $m\rightarrow \infty$ for all $\theta\in K$.
\end{lemma}
\begin{theorem}
\label{thm:FDEconvergence_Banach}
{\bf (Convergence of cylindrical approximations to FDEs)}
Suppose that the initial value problem \eqref{thefde} is 
well-posed in the time interval $[0,T]$ ($T$ finite), 
and that $\mathcal{L}([\theta])\in 
\mathcal{C}(\mathcal{F})$ generates a 
strongly continuous semigroup in $[0,T]$.
Then the FDE approximation \eqref{thepde} is 
stable and consistent (in the sense 
of Definitions \ref{def:compatibility} and \ref{def:stability})
in a compact subset $K$ of a real  Banach 
space $X$ admitting a basis if and only if it is convergent, i.e.,
\begin{equation}
\max_{t\in [0,T]}\max_{\theta\in K}
\left| F_m([\theta],t) - F([\theta],t) \right|
\rightarrow 0 
\label{convergence_Banach}
\end{equation}
as $m\rightarrow \infty$, provided 
$F_m([\theta],0)\rightarrow F_0([\theta])$. 
\end{theorem}
The proof of this theorem can be found 
in \cite[p. 210]{engel1999one}. 
We emphasize that the sequence of steps to 
prove convergence of functional approximations 
to FDEs in Banach spaces admitting a basis is 
the same a)-c) listed after 
Theorem \ref{thm:FDEconvergence}.}

\section{Numerical examples}
\label{sec:numerics}

In this section we provide numerical demonstrations 
of the approximation theorems we developed 
for nonlinear functionals and functional differential 
equations. To this end, we consider the function space 
defined by the following closure 
of a Sobolev sphere with raidius $\rho$
\begin{equation}
K=\overline{\left\{\theta\in H_p^{s}([0,2\pi]): 
\left\|\theta\right\|_{H_p^s} \leq \rho \right\}}\subseteq L_p^2([0,2\pi]). 
\label{periodicFspace}
\end{equation}
We have seen in section \ref{sec:functionalonmetric} 
that $K$ is a convex compact subset of $L_p^2([0,2\pi])$. 
Hence, any real-valued continuous functional 
$F([\theta])$ defined on $K$ can be represented 
as the limit of a uniformly convergent sequence 
of functionals of the form $F([P_m \theta])$, 
where $P_m$ is the projection operator \eqref{proj_op}.
We can sample elements from \eqref{periodicFspace} 
by taking a truncated Fourier series of the form  
\begin{equation}
\theta(x)=\sum_{k=-N}^N c_ke^{ikx},\qquad c_k=c_{-k}^*,
\label{FS}
\end{equation}
and then choosing the modulus of the 
complex numbers $\{c_0,\ldots,c_N\}$ within an 
ellipsoid in $\mathbb{R}^{N+1}$. In fact, we have 
\begin{align}
\left\|\theta\right\|^2_{H_p^s}=&
\sum_{k=-N}^N (1+k^2+k^4+\cdots k^{2s}) 
\left| c_k \right|^2 \nonumber\\
=& c_0^2 + 2\sum_{k=1}^N (1+k^2+k^4+\cdots k^{2s}) \left|c_k\right|^2.
\label{HNRM}
\end{align}
Hence, the condition $\left\|\theta\right\|^2_{H_p^s}\leq \rho^2$ implies that 
\begin{equation}
c_0^2 + 2\sum_{k=1}^N (1+k^2+k^4+\cdots k^{2s}) \left|c_k\right|^2\leq \rho^2,
\end{equation}
which defines the interior of an ellipsoid in the 
variables $\{|c_0|,\ldots,|c_N|\}$.

\subsection{Generation of test functions with prescribed Fourier spectrum}
The decay rate of the modulus of the Fourier coefficients 
$\left|c_k\right|$ in the series expansion 
\eqref{FS} is related to the degree 
of smoothness of $\theta$, i.e., the value of 
$s$ in \eqref{periodicFspace} (see \cite[\S 2]{spectral}). 
Hence, by sampling $\theta$ from a space of periodic 
functions with a prescribed spectral decay $\left|c_k\right|$ 
we can study the effects of  the regularity parameter 
$s$ in \eqref{periodicFspace} on 
the rate of convergence of the nonlinear functional 
approximations we developed in section 
\ref{sec:finiteDimapprox}, section \ref{sec:F} and 
section \ref{sec:FD}.
To sample test functions from \eqref{periodicFspace}, 
we represent $c_k$ in \eqref{FS} in polar form, prescribe 
the decay of the amplitudes 
$|c_k|$ ($k\geq 0$) and introduce a uniformly distributed 
random shift $\vartheta_k\in[0,2\pi]$ subject to the constraint 
$\vartheta_k=-\vartheta_{-k}$. This yields
\begin{equation}
\theta(x) = c_0 + \sum_{\substack{k=-N\\k\neq 0}}^N |c_k| e^{i(kx+\vartheta_k)}, \qquad c_k=c_{-k}^*.
\label{samples}
\end{equation}
We study two types of decay rates of the 
Fourier spectrum.  The first is a 
power-law decay of the form
\begin{equation}
c_0=a(0), \qquad |c_k| = \frac{a(k)}{k^\alpha} \qquad 
\text{(algebraic decay)}, 
\label{powerlaw}
\end{equation}
where $\alpha\geq 1$ and $k=1,\ldots, N$.
In equation \eqref{powerlaw} $a(0)$ is a uniformly 
distributed random variable in $[-10,10]$ 
and $\{a(1),\ldots, a(N)\}$ is a sequence of 
i.i.d. uniformly distributed random variables in $[0,10]$. 
The algebraic decay \eqref{powerlaw} defines 
functions in a Sobolev sphere \eqref{periodicFspace} 
with index $s=\alpha$. The radius of such sphere 
can be computed by substuting \eqref{powerlaw}
 into \eqref{HNRM}, and 
then evaluating the supremum.
The second power spectrum we consider 
has an exponential decay of the form
\begin{equation}
c_0=b(0), \qquad |c_k| = \frac{b(k)}{\beta^k} \qquad 
\text{(exponential decay)},
\label{exponential}
\end{equation}
where $\beta > 1$ and $k=1,\ldots, N$.
The random sequence $\{b(0),b(1),\cdots, b(N)\}$ has 
the same properties as the sequence 
$\{a(0),a(1),\ldots, a(N)\}$ in \eqref{powerlaw}. 
The spectrum \eqref{exponential} defines 
 functions in a Sobolev sphere \eqref{periodicFspace}
with index $s\rightarrow \infty$.

In Figure \ref{fig:inputsignals} we plot one sample of the 
random spectra \eqref{powerlaw} and \eqref{exponential}, 
together with the corresponding sample functions 
\eqref{samples} for $N=500$,
$\alpha \in \{1.5, 2, 2.5, 3\}$ and 
$\beta \in \{1.2, 1.5, 2, 3\}$.  
In the numerical examples presented hereafter we 
choose $N$ large enough so that the contribution of 
the tail of the spectrum is negligible in the series expansion 
\eqref{samples}. This allows us to generate
highly accurate approximations of $\theta$ in 
the space \eqref{periodicFspace}, 
which will then be projected onto a lower-dimensional 
subspace generated by a second trigonometric basis.
\begin{figure}[t]
\centerline{\footnotesize\hspace{0.4cm}Power law decay $|c_k|\sim 1/k^\alpha$ \hspace{4.5cm} Exponential decay $|c_k|\sim 1/\beta^k$}
\centerline{
\includegraphics[height=6cm]{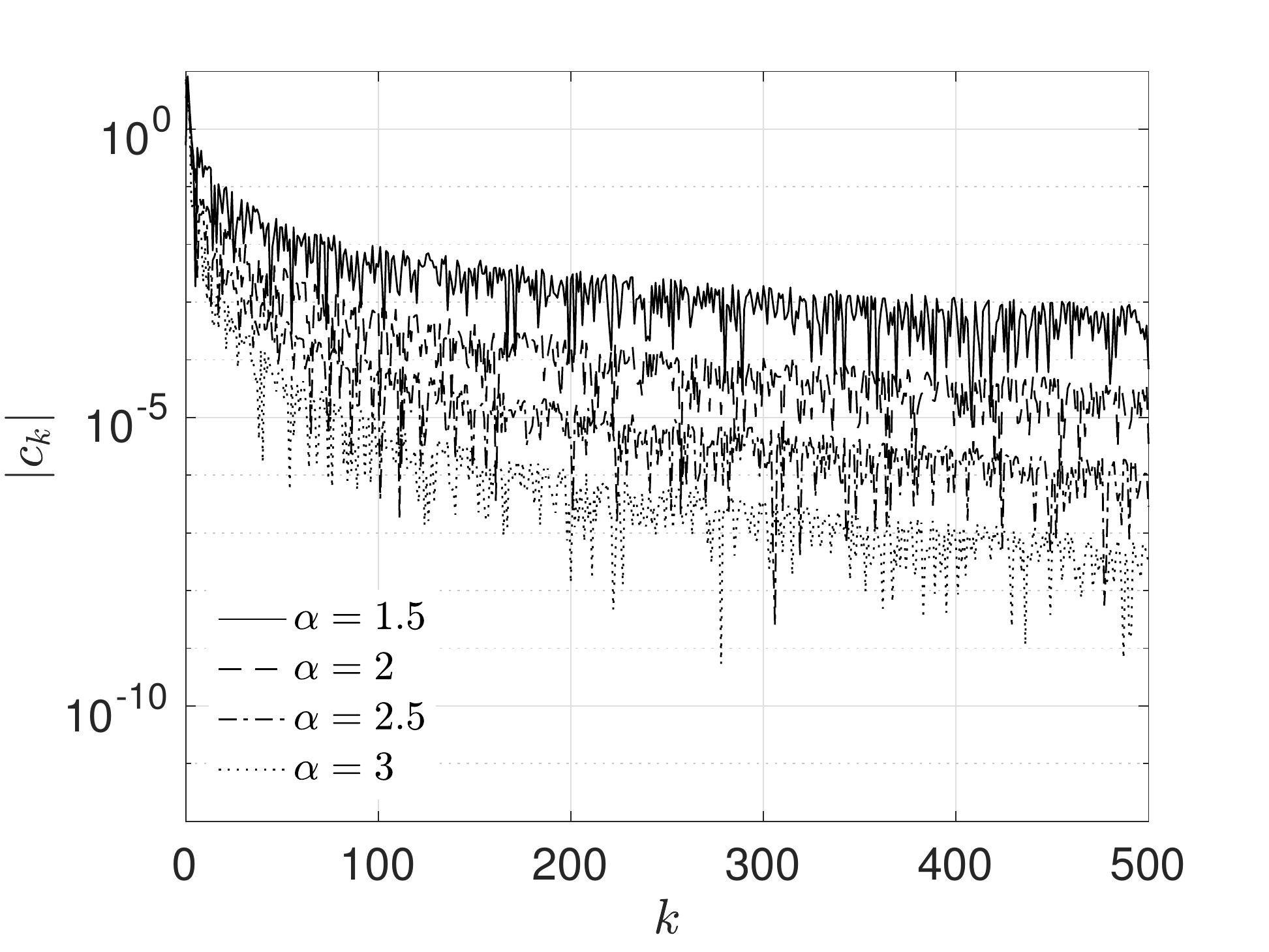}
\includegraphics[height=6cm]{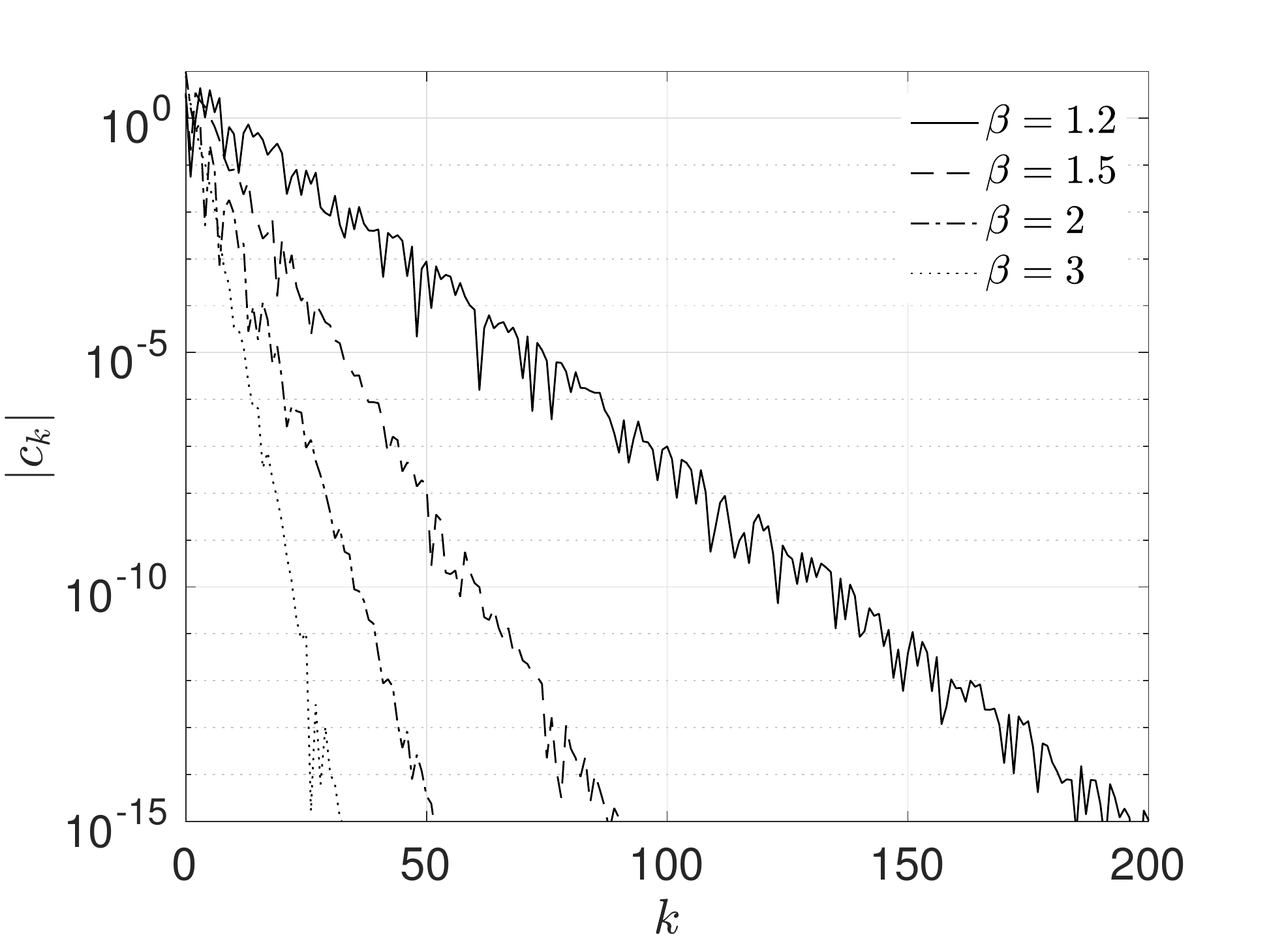}
}\hspace{0.1cm}
\centerline{
\includegraphics[height=6cm]{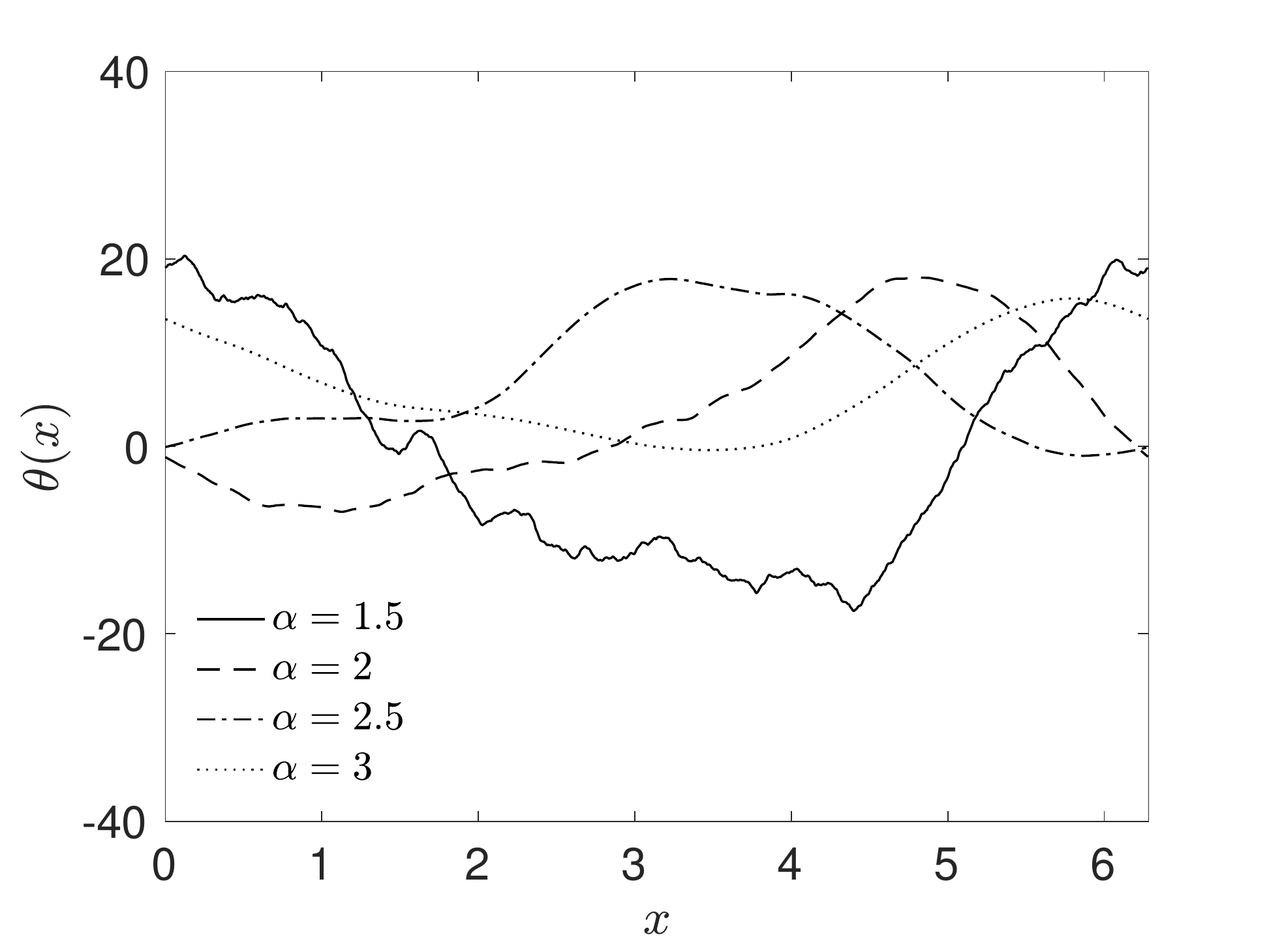}
\includegraphics[height=6cm]{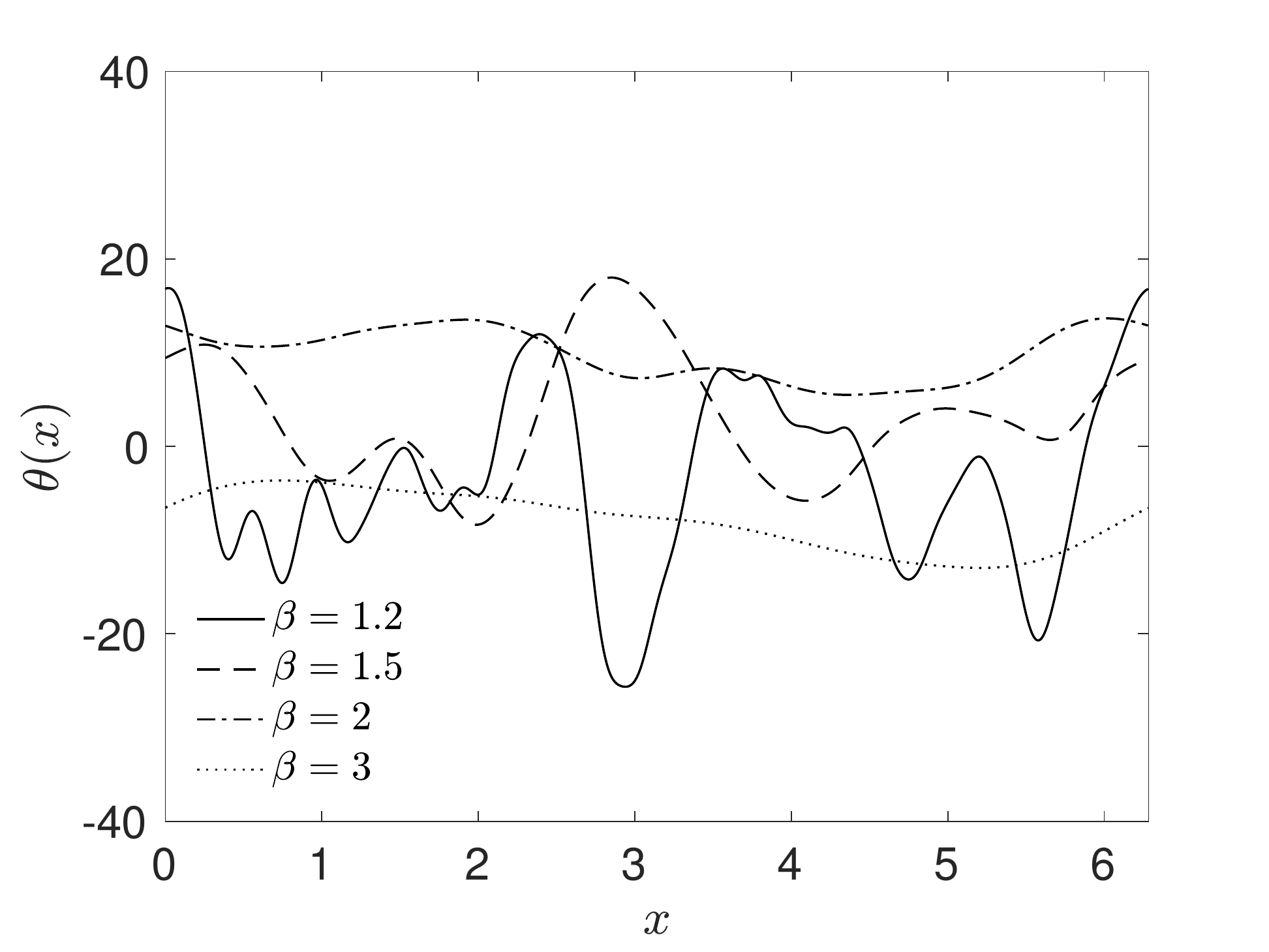}
}
\caption{Random spectra \eqref{powerlaw} and \eqref{exponential}
and corresponding sample functions \eqref{samples} 
for $N=500$.}
\label{fig:inputsignals}
\end{figure}
Specifically, we chose the following orthonormal basis  
consisting of discrete trigonometric polynomials 
\cite[p. 29]{spectral}
\begin{equation}
\displaystyle
\varphi_k(x)= \displaystyle \frac{1}{\sqrt{2\pi(m+1)}}
\frac{\displaystyle \sin\left((m+1)\frac{x-x_{k}}{2}\right)}
{\displaystyle \sin\left(\frac{x-x_{k}}{2}\right)},\qquad 
x_k=\frac{2\pi}{m+1}k,\qquad k=0,\ldots, m\quad \text{($m$ even)}.
\label{trigono}
\end{equation}
which yields the projection operator
\begin{align}
P_{m}\theta = \sum_{k=0}^{m} a_k \varphi_k(x), \qquad a_k = (\theta, \varphi_k)_{L_p^2([0,2\pi])}, \qquad k = 0, 1, \ldots, m.
\label{theta_m}
\end{align}
As is well known,
if the first $(s-1)$ derivatives of $\theta$ 
are all continuous, and if the $s$-th derivative 
is in $L_p^2([0,2\pi])$  
then the $L_p^2([0,2\pi])$ distance  between 
$\theta$ and $P_m\theta$ as defined in \eqref{theta_m} 
decays as $1/m^s$. On the other hand, if $\theta$ is 
of class $C^{\infty}$ then then $P_m\theta$ converges 
to $\theta$ exponentially fast 
in $L_p^2([0,2\pi])$ (see \cite[\S 2.3]{spectral}).

\subsection{Approximation of nonlinear functionals}
\label{sec:nlinF}
Consider the nonlinear functional
\begin{equation}
F([\theta])=\int_0^{2\pi} \sin(x)\sin(\theta(x))^2 dx.
\label{Ftest}
\end{equation}
The Fr\'echet differential of $F([\theta])$ is given by
\begin{equation}
F'([\theta])\eta = \int_0^{2\pi} \sin(x)\sin(2\theta(x))\eta(x) dx, 
\label{GD}
\end{equation}
which is  a linear operator in $\eta$. 
We have shown in 
section \ref{sec:functionalonmetric} that $F'([\theta])$ 
is compact in the function space \eqref{periodicFspace}
and therefore it is bounded and continuous\footnote{Recall 
that a linear functional in a Hilbert space is bounded if 
and only if it is continuous.}. 
In fact, for all $\theta\in K$ and 
$\eta\in L^2_p([0,2\pi])$ it follows from \eqref{GD} that
\begin{equation}
 \left|F'([\theta]) \eta\right|\leq 
 \left\|\sin(x)\sin(2\theta)\right\|_{L_p^2([0,2\pi])}
 \left\|\eta\right\|_{L_p^2([0,2\pi])}\quad \Rightarrow \quad 
 \left\|F'([\theta])\right\| \leq \sqrt{\pi}.
\end{equation}
 Plugging this result into the mean value Theorem \ref{MVT} yields  the spectral convergence result
\begin{equation}
\left| F([\theta])-F([P_m\theta])\right|    \leq \sqrt{\pi}
\left\|\theta-P_m\theta\right\|_{L_p^2([0,2\pi])}                \leq 
\frac{\sqrt{\pi} C}{m^s}\left\|\theta \right\|_{H_p^s([0,2\pi])}   \leq 
\frac{\sqrt{\pi} C\rho}{m^s}, \qquad \forall \theta \in K.
\label{convr}
\end{equation}  
The last two inequalities follow from well-known Fourier
series approximation theory \cite[p. 42]{spectral}, 
and from the fact that $\theta$ is in the closure of a
Sobolev sphere with radius bounded by $\rho$.

Next, we determine the convergence rate of 
the first-order Fr\'echet and functional derivative 
approximations. To this end, we notice that 
the second-order Fr\'echet derivative 
of \eqref{GD}, i.e.,  
\begin{equation}
F''([\theta])\eta\psi = \int_{0}^{2\pi} 2\sin(x)\cos(2\theta(x)) \psi(x)\eta(x) dx, 
\end{equation}
is a continuous bilinear operator 
on the compact set $K \times K$. Therefore, by 
equation \eqref{FD2}, the first-order Fr\'echet derivative 
must converge at the same rate as \eqref{convr}. 
The first-order functional derivative of $F$, i.e., the kernel 
of the integral operator \eqref{GD} is 
\begin{equation}
\frac{\delta F([\theta])}{\delta \theta(x)}= \sin(x)\sin(2\theta(x)).
\label{exactpDF}
\end{equation}
As easily seen, if we evaluate \eqref{exactpDF} at  $P_m\theta$
we obtain the approximated functional derivative 
\begin{equation}
\frac{\delta F([P_m\theta])}{\delta \theta(x)}= \sin(x)\sin\left(2
\sum_{k=1}^m a_k\varphi_k(x)\right),\qquad a_k=(\theta,\varphi_k)_{L_p^2([0,2\pi])},
\label{123}
\end{equation}
which is an element of $L_p^2([0,2\pi])$ that converges 
to \eqref{exactpDF} uniformly in $\theta\in K$ in 
the $L_p^2([0,2\pi])$ norm.  This is because 
bounded and continuous functions 
such as $\sin(x)$ preserve $L^2$ convergence 
under composition (see \cite[Theorem 7]{Bartle}).  
This result is also in agreement with Lemma \ref{FunctionalDerivativeApprox}.
Hereafter we provide a numerical verification 
of the convergence rate we just predicted. 
To this end, in Figure \ref{fig:functionals} we plot 
\begin{equation}
\epsilon_0(m) = \sup_{\theta\in K} |F([\theta])-F([P_m\theta])| 
\label{epsilon0}
\end{equation}
versus $m$. The error $\epsilon_0(m)$ is computed numerically 
for each $m$ by taking the maximum over $10^3$ 
sample functions of the form \eqref{samples}, 
with $N=1000$, and different spectra of the 
form \eqref{powerlaw} and \eqref{exponential}
(see Figure \ref{fig:inputsignals}).
\begin{figure}
\centerline{\footnotesize\hspace{0.6cm}Power law decay 
$|c_k|\sim 1/k^\alpha$ \hspace{5cm} Exponential 
decay $|c_k|\sim 1/\beta^k$}
\centerline{
\includegraphics[height=6.3cm]{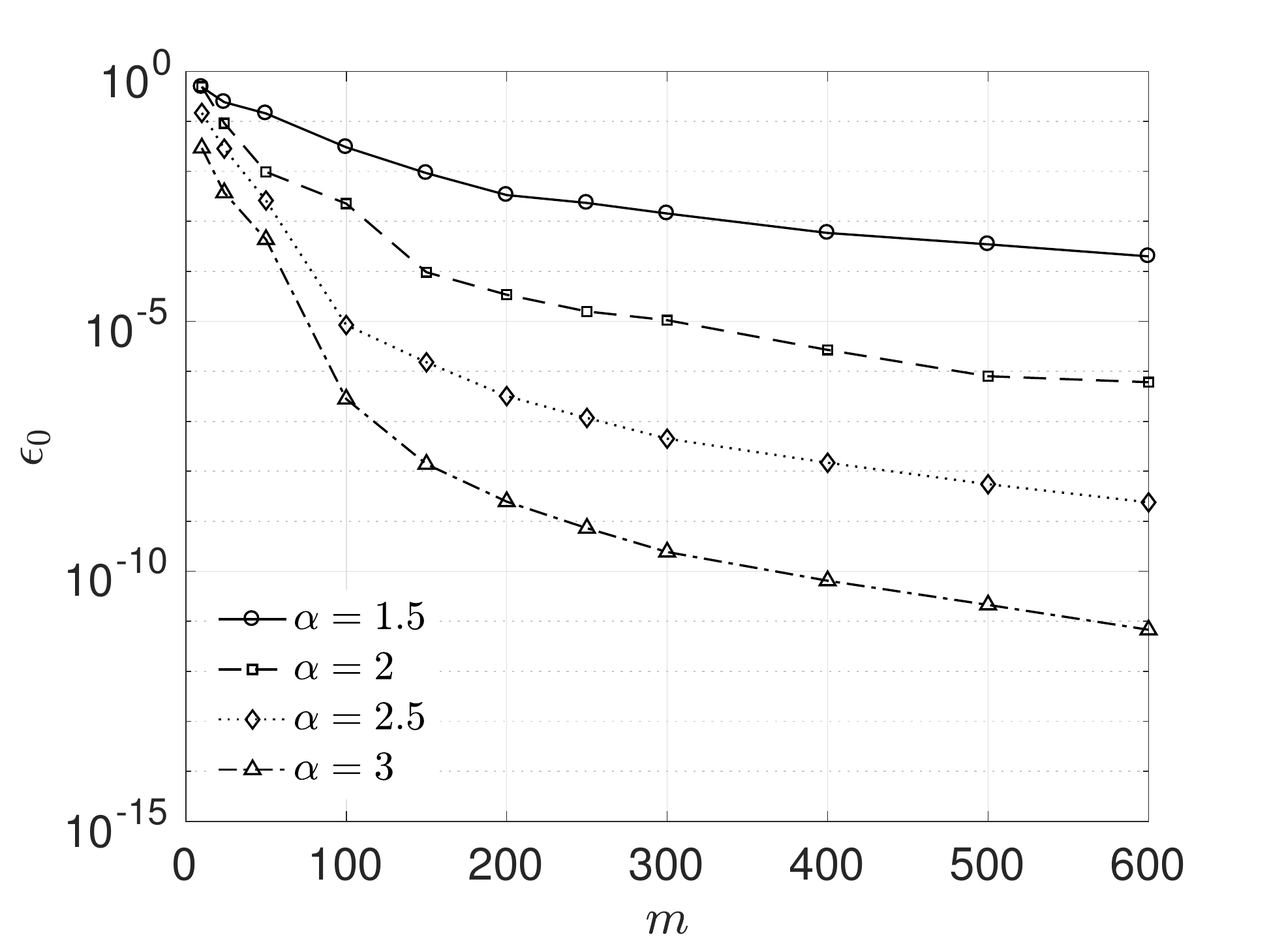}
\includegraphics[height=6.3cm]{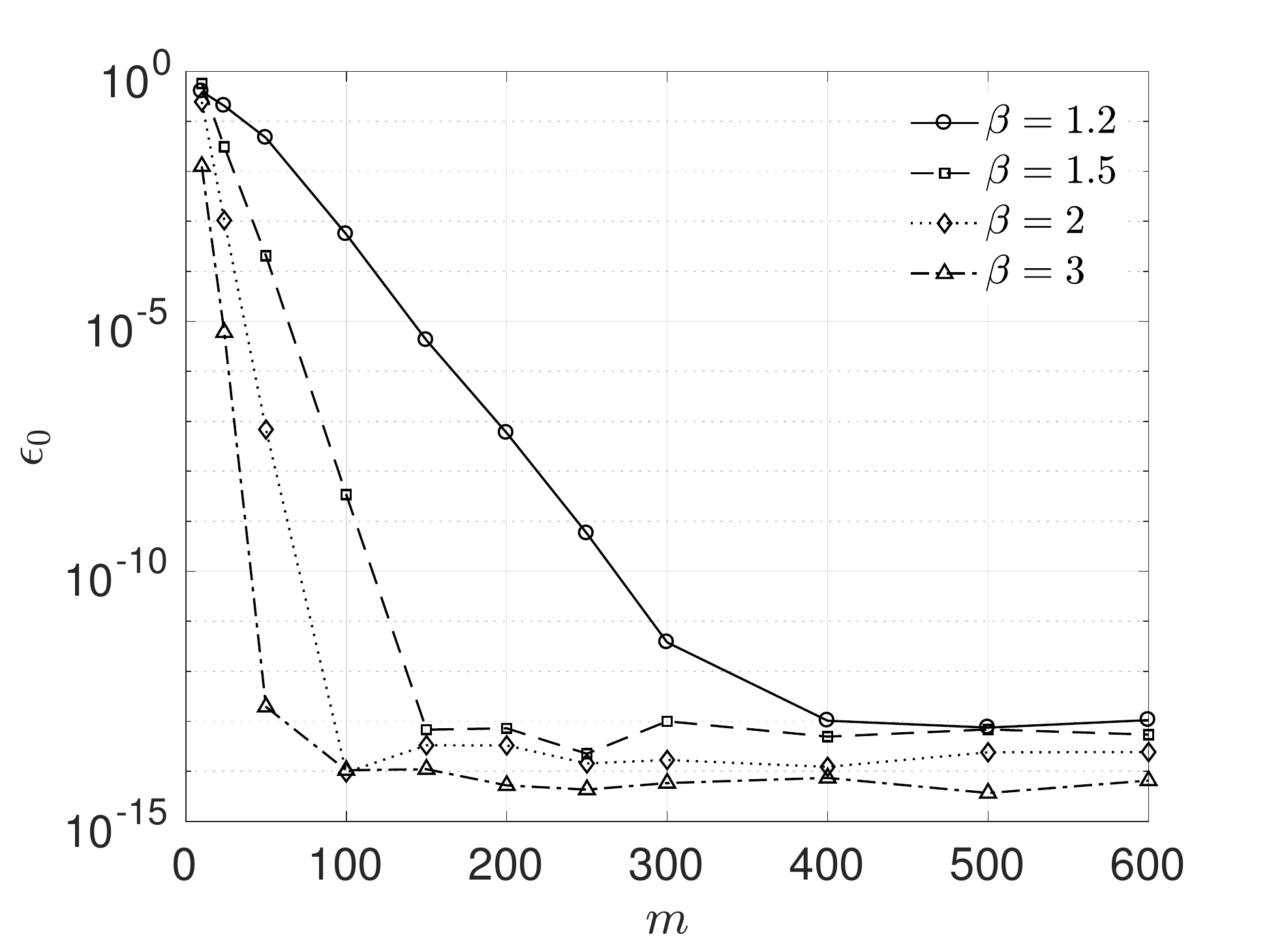}
}
\caption{Functional approximation error \eqref{epsilon0} versus 
the number of Fourier modes in \eqref{theta_m} 
for functions $\theta$ with spectra \eqref{powerlaw} 
(power law decay) and \eqref{exponential} (exponential decay). 
The functional \eqref{Ftest}
is continuously Fr\'echet differentiable. 
Therefore, by Lemma \ref{LMVT} we have that 
$F([P_m\theta])$ 
converges to $F([\theta])$ at the same 
rate as $P_m\theta$ converges to $\theta$.}
\label{fig:functionals}
\end{figure}
The error in the Fr\'echet derivative is defined as
\begin{equation}
\epsilon_1(m) = \sup_{\substack{\eta,\theta\in K\\\eta\neq 0}} 
\frac{\left| F'([\theta])\eta-
F'([P_m \theta])\eta \right|}{\left\|\eta\right\|_{L_p^2([0,2\pi])}},
\label{epsilon1}
\end{equation}
and is computed as follows: 
for each given $\theta$, we determine $P_m\theta$ and then 
approximate the supremum over $\eta$ using $10^3$ 
sample functions $\eta$. This is done for $10^3$ functions 
$\theta$ sampled from $K$ as before. Notice that 
$\eta\in K$ has the same form as $\theta$ and 
therefore it is taken from the same 
ensemble as $\theta$ is taken from. 
The results of our calculations are shown in 
Figure \ref{fig:functional_derivatives}. As expected, 
the approximated functional derivative 
$F'([P_m\theta])$ converges to $F'([\theta])$ at the 
same rate as $P_m\theta$ converges to $\theta$ in 
$L^2_p([0,2\pi])$. The reason is that 
the Fr\'echet derivative 
\eqref{GD} is continuously 
Fr\'echet differentiable\footnote{The 
functional \eqref{Ftest} admits 
continuous Fr\'echet derivatives to any 
desired. In particular, we have 
\begin{align}
F''([\theta])\eta_1\eta_2 = &\,\,
2\int_{0}^{2\pi} \sin(x)\cos(2\theta(x))
\eta_1(x)\eta_2(x)dx,\nonumber\\ 
F'''([\theta])\eta_1\eta_2\eta_3 = &
-4\int_{0}^{2\pi} \sin(x)\sin(2\theta(x))
\eta_1(x)\eta_2(x)\eta_3(x)dx.\nonumber
\end{align}
This implies that the mean value formula 
\eqref{FD2} can applied to any of 
the Fr\'echet derivatives, by simply 
redefining the operator norm appearing 
at the right hand side of the inequality.}, 
and therefore the mean value formula 
\eqref{FD2} holds.
\begin{figure}
\centerline{\footnotesize\hspace{0.6cm}Power law decay 
$|c_k|\sim 1/k^\alpha$ \hspace{5cm} Exponential 
decay $|c_k|\sim 1/\beta^k$}
\centerline{
\includegraphics[height=6.3cm]{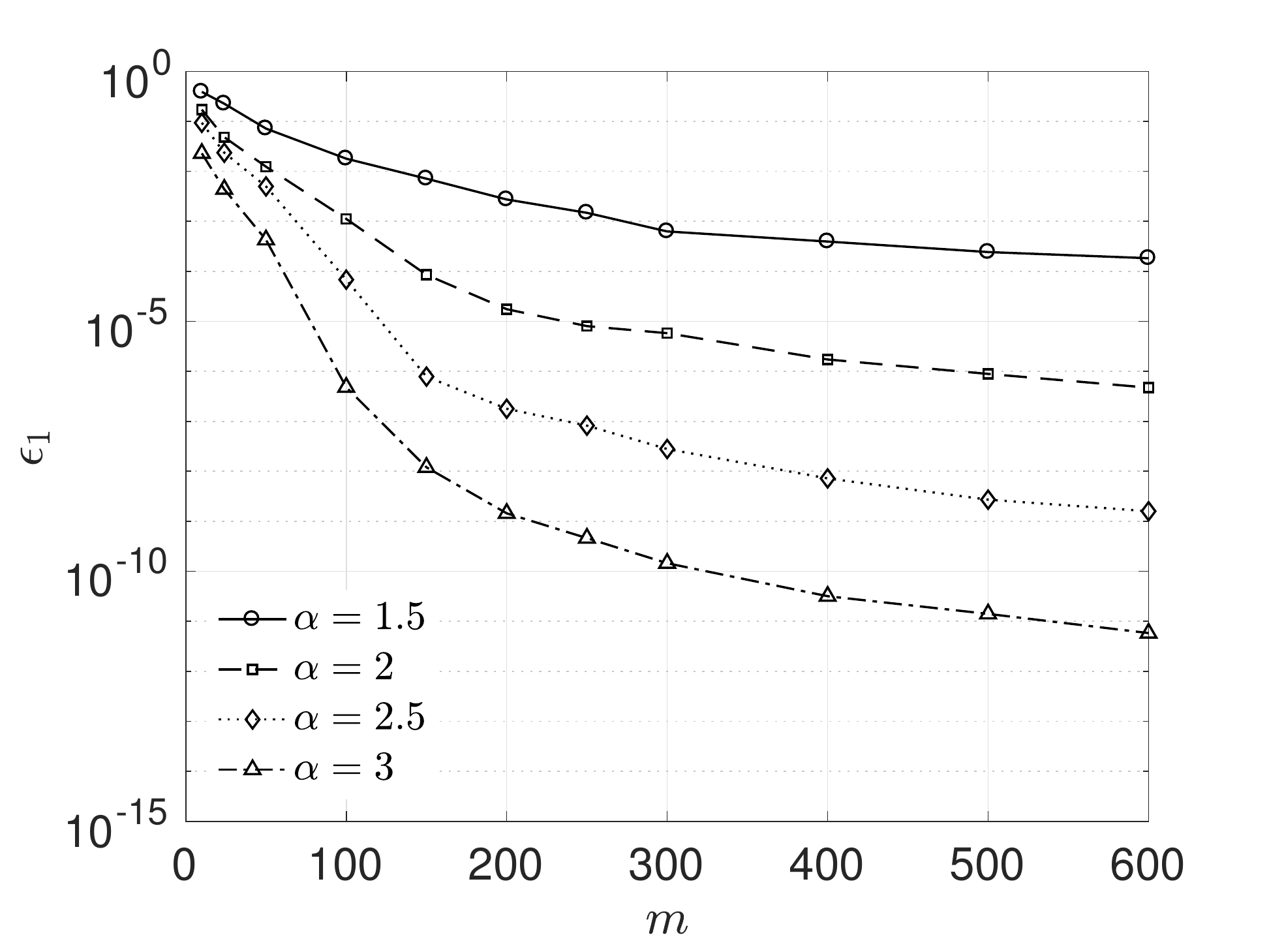}
\includegraphics[height=6.3cm]{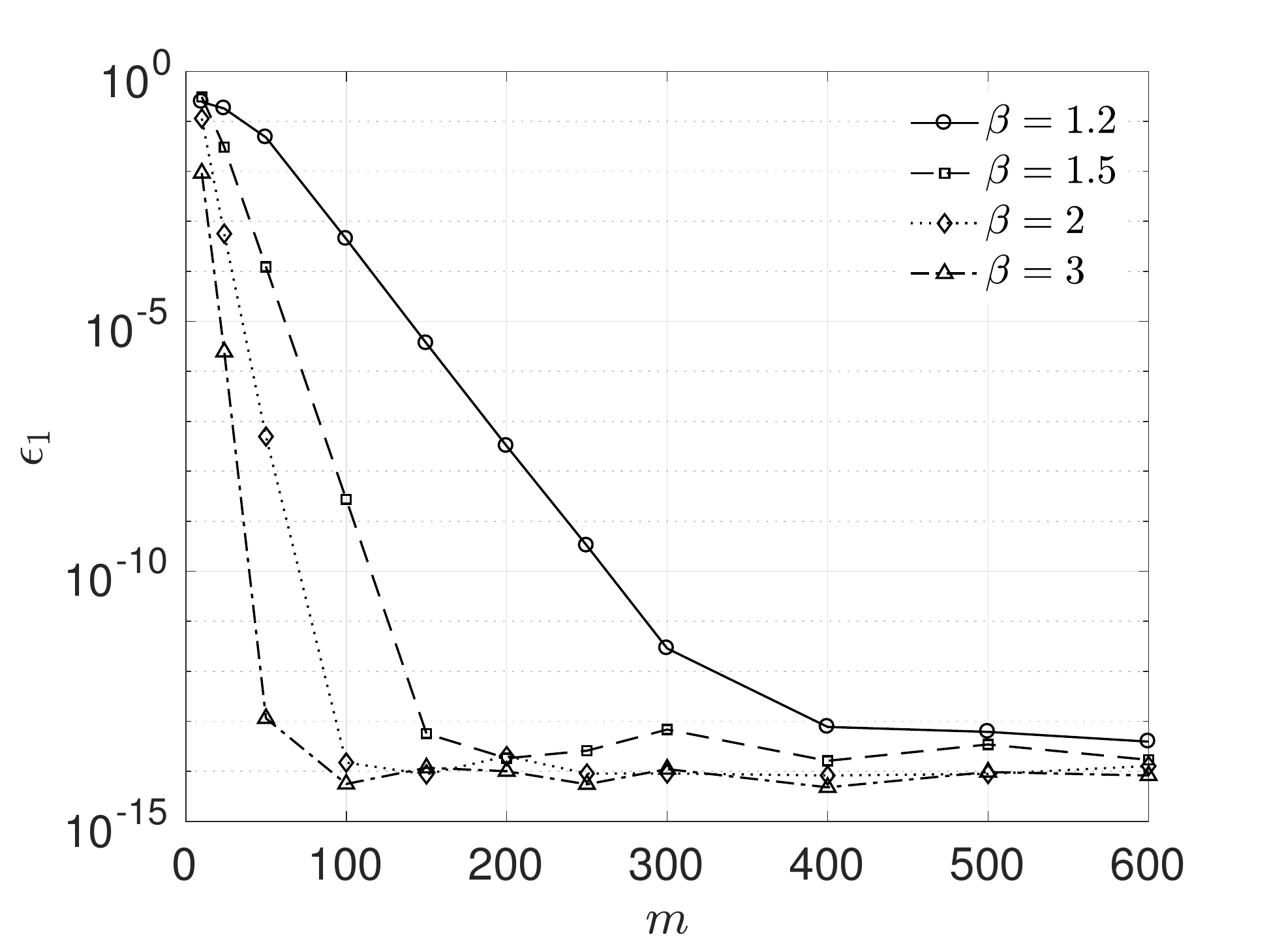}
}
\caption{Fr\'echet derivative approximation 
error \eqref{epsilon1} versus the number of 
Fourier modes in \eqref{theta_m}, and for 
test functions $\theta$ with spectra \eqref{powerlaw} 
(power law decay) and \eqref{exponential} 
(exponential decay). Note that, as expected, the 
approximated Fr\'echet  derivative $F'([P_m \theta])$ 
converges to $F'([\theta])$ at the same 
rate as $P_m\theta$ converges to $\theta$. 
The reason is that the functional 
$F([\theta])$ is continuously Fr\'echet differentiable 
to any desired order. Hence, the mean value formula 
\eqref{FD2} holds. This is also the reason 
why the convergence plots are nearly identical to 
those in Figure \ref{fig:functionals} 
(compare \eqref{FD2} with \eqref{estimateMVT}).}
\label{fig:functional_derivatives}
\end{figure}

\subsection{Approximation of functional differential equations}
\label{sec:FDEconvergenceNumerics}

In this section we provide a simple example of 
convergence analysis that shows how fast 
the solution of the multivariate PDE 
\eqref{example2C} converges to the solution of 
the FDE \eqref{example2} as we send $m$ to infinity. 
To this end, we first examine the analytical solution 
of the FDE \eqref{example2}.

\subsubsection{Analytical solution}
It was shown in \cite[p. 76]{venturi2018numerical} 
that the analytical solution of the FDE \eqref{example2C} in the 
function space \eqref{periodicFspace} is
\begin{equation}
F([\theta],t)= F_{0}([\theta(x-t)]).
\end{equation}
Clearly, if $F_0$ is invariant under 
translations, i.e., if $F_0([\theta(x-t)])=F_0([\theta(x)])$, 
then $F([\theta],t)= F_{0}([\theta])$, i.e., the 
solution is constantly equal to the initial condition 
$F_0([\theta])$ at each time. Examples of such 
translation-invariant functionals  are
\begin{equation}
\widehat{F}_{0}([\theta]) = 
\exp\left[-\int_{0}^{2\pi}\theta(x)^2dx\right], 
\quad \text{and}\quad
\widetilde{F}_{0}([\theta]) = \int_{0}^{2\pi}
\sin\left(\theta(x)\right)^2dx.
\end{equation}
On the other hand, the initial condition 
\begin{equation}
F_{0}([\theta]) = \int_{0}^{2\pi}\sin(x)
\sin\left(\theta(x)\right)^2dx
\label{IC}
\end{equation}
is not translation-invariant. The solution to the initial 
value problem \eqref{example2}, with $F_0$ 
given in \eqref{IC}, is 
\begin{equation}
F([\theta],t) = \int_{0}^{2\pi}\sin(x)
\sin\left(\theta(x-t)\right)^2dx,
\label{FDEsol}
\end{equation}
which is periodic in $t$ with period $2\pi$.
It is easy to verify by direct calculation that \eqref{FDEsol} 
is indeed a solution to \eqref{example2}. To this end, let us 
define $\partial_x=\partial/\partial x$.
We begin by noting that 
\begin{equation}
\theta(x-t)=e^{-t\partial_x}\theta(x).
\label{exptransf}
\end{equation}
The first-order functional derivative of \eqref{FDEsol} 
is obtained by analyzing its  Fr\'echet differential 
\begin{align}
dF_\eta([\theta],t)=&\int_{0}^{2\pi} 2\sin(x)\sin(e^{-t\partial_x}\theta) \cos(e^{-t\partial_x}\theta) e^{-t\partial_x}\eta dx \nonumber\\
=& \int_{0}^{2\pi}  e^{t\partial_x} \left[\sin(x)\sin(2e^{-t\partial_x}\theta) \right]\eta dx. 
\end{align}
Here we utilized the fact that the operator adjoint of the 
semigroup $e^{-t\partial_x}$ relative to standard 
$L_p^2([0,2\pi])$ inner product is $e^{t\partial_x}$.
Hence, the first-order functional derivative of \eqref{FDEsol} is  
\begin{equation}
\frac{\delta F([\theta],t)}{\delta \theta(x)} = e^{t\partial_x}\left[ \sin(x)\sin(2e^{-t\partial_x}\theta)\right].
\end{equation}
Using again the fact that $\partial_x$ is skew-symmetric 
relative to the $L_p^2([0,2\pi])$ 
inner product we obtain   
\begin{align}
\int_{0}^{2\pi} \frac{\partial }{\partial x}
\left(\frac{\delta F([\theta],t)}{\delta \theta(x)}\right)\theta(x)dx = &\int_{0}^{2\pi}  \partial_x e^{t\partial_x} \left[\sin(x)\sin(2e^{-t\partial_x}\theta)\right]\theta dx  \nonumber\\
 = &-\int_{0}^{2\pi}  \sin(x)\sin(2e^{-t\partial_x}\theta)e^{-t\partial_x} \partial_x\theta dx \nonumber\\
 = &-\int_{0}^{2\pi}  \sin(x)\sin(2\theta(x-t)) \partial_x\theta(x-t) dx.
 \label{dFdt}
\end{align}
On the other hand, a temporal differentiation 
of \eqref{FDEsol} yields
\begin{equation}
\frac{\partial F([\theta],t)}{\partial t} = 
\int_{0}^{2\pi} \sin(x)\sin(2\theta(x-t))\partial_t \theta(x-t) dx.
\label{dFdt1}
\end{equation}
By setting the equality between \eqref{dFdt} and 
\eqref{dFdt1} we conclude that \eqref{FDEsol} is 
a solution to \eqref{example2} if and only if
\begin{equation}
\frac{\partial \theta(x-t)}{\partial t}+
\frac{\partial \theta(x-t)}{\partial x}=0,
\label{pdeADV}
\end{equation}
which is clearly an identity, given \eqref{exptransf}. 
This proof can be generalized to arbitrary 
Fr\'echet differentiable initial conditions $F_0$.

\subsubsection{FDE approximation and convergence analysis}
\label{sec:convFDE}
We have seen in section \ref{sec:FDE} that 
the cylindrical approximation of the FDE 
\eqref{example2} yields the multivariate PDE 
\eqref{example2C}. By using integration by parts 
it can be shown that the matrix of coefficients
\begin{equation}
C_{jk}=\int_0^{2pi} \varphi_j(x)
\frac{\partial \varphi_k(x)}{\partial x}dx
\label{Cij}
\end{equation}
is skew-symmetric since the basis 
functions $\varphi_j$ are periodic. 
The initial condition appearing in \eqref{example2C} 
is obtained by evaluating \eqref{IC} on the 
range of $P_m$. This  yields the cylindrical functional 
\begin{equation}
f_0(a) = \int_{0}^{2\pi}\sin(x)
\sin\left(\sum_{k=0}^m a_j\varphi_j(x)\right)^2dx,\qquad 
a_j=(\theta,\varphi_k)_{L_p^2([0,2\pi])}.
\label{HDIC}
\end{equation}
The solution to the initial value problem 
\eqref{example2C} with initial condition given in 
\eqref{HDIC} is obtained as 
\begin{equation}
f(a,t) = f_0\left(e^{tC} a\right), \qquad a=[a_0,\ldots, a_m]^T.
\label{PDEsol}
\end{equation}
We have seen in section \ref{sec:nlinF} that 
\eqref{HDIC} converges uniformly 
to $F_0([\theta])$ as as $m$ goes to infinity 
at the same rate as $\left\|\theta - P_m\theta\right\|_{L_p^2([0,2\pi])}$
goes to zero. We also know that the residual of the 
finite-dimensional PDE approximation \eqref{example2C} 
goes to zero as we send $m$ to infinity (Example 1 
in section \ref{sec:FDE}), and that 
\eqref{example2C}-\eqref{HDIC} 
is stable in the $L^{\infty}$ norm (Example 2 in 
section \ref{thm:consistencyORDER}). By 
Theorem \ref{thm:FDEconvergence} this is sufficient 
to guarantee that \eqref{PDEsol} converges uniformly 
in $\theta$ to the FDE solution
\eqref{FDEsol} as we increase $m$ 
(Theorem \ref{thm:FDEconvergence}).
Hereafter we calculate the convergence rate of 
such approximation, and show that it can be 
exponential depending on degree of
smoothness of $\theta\in K$, which is measured 
by the index $s$ in \eqref{periodicFspace}.
To this end, we begin with 
\begin{align}
\left| F([\theta],t) - 
f(a_0,\ldots, a_m,t)\right|^2 &= 
\left| F_0([e^{-t\partial_x}\theta]) - 
f_0(e^{tC}a)\right|^2\nonumber\\
&=\left|\int_{0}^{2\pi} 
\sin(x)\left[\sin^2\left(e^{-t\partial_x}\theta(x)\right)-
\sin^2\left(\sum_{j=0}^m [e^{tC}a]_j\varphi_j(x)\right)\right]dx
\right|^2. \label{f1}
\end{align}
Recall that for any $a,b\in \mathbb{R}$ 
we have
\begin{equation}
\sin^2(a)-\sin^2(b)= \sin(a+b)\sin(a-b).
\end{equation}
Hence, from equation \eqref{f1} it follows that 
\begin{align}
\left| F([\theta],t) - 
f(a_0,\ldots, a_m,t)\right|^2&\leq 
\int_{0}^{2\pi}\left|\sin\left(e^{-t\partial_x}\theta(x)-\sum_{j=0}^m [e^{tC}a]_j\varphi_j(x)\right)\right|^2dx\nonumber\\
&\leq\int_{0}^{2\pi}\left| e^{-t\partial_x}\theta(x)-\sum_{j=0}^m [e^{tC}a]_j\varphi_j(x)\right|^2dx\nonumber\\
&=\left\| e^{-t\partial_x}\theta(x)-\sum_{j=0}^m [e^{tC}a]_j\varphi_j(x)\right\|^2_{L_p^2([0,2\pi])}.
\label{f2}
\end{align}
At this point, we recall that  
$\theta(x,t)=e^{-t\partial_x}\theta(x)$ is the 
exact solution to the advection equation  \eqref{pdeADV},
while the function $\displaystyle\theta_m(x,t)=\sum_{j=0}^m 
\left[e^{tC}a\right]_j\varphi_j(x)$
is the solution to the Fourier-Galerkin 
discretization of \eqref{pdeADV} 
\begin{equation}
\frac{da_j(t)}{dt}-\sum_{k=1}^m a_k C_{jk}=0.
\label{fgs}
\end{equation}
It is well-known that the 
Galerkin scheme \eqref{fgs} is stable in 
the $L_p^2([0,2\pi])$ norm (see, e.g.,  \cite[\S 6.1.1]{Canuto}), 
and that the solution $\theta_m(x,t)$ converges 
to $\theta(x,t)$ at a rate that depends only on the 
smoothness of $\theta(x,0)$ (initial condition). 
This implies that 
\begin{equation}
\left| F([\theta],t) - 
f(a_0,\ldots, a_m,t)\right|\leq 
\frac{C}{m^s}\left\|\theta \right\|_{H_p^s}
\leq \frac{C\rho}{m^s},
\label{spectralCONV}
\end{equation}
where the parameter $s$ measures the 
regularity of $\theta$. 
If $\theta$ is infinitely differentiable, 
then $f(a_0,\ldots, a_m,t)$ converges to 
$F([\theta],t)$ exponentially fast in $m$.
To validate \eqref{spectralCONV} numerically, 
in Figure \ref{fig:FDE} we plot the error 
\begin{equation}
\epsilon_0(m,t) = \sup_{\theta\in K} \left| F([\theta],t)-
f(a_0,\cdots,a_m,t)\right|
\end{equation}
at time $t=\pi$ in the case where $\theta$ has power law or 
exponential decaying Fourier coefficients. It is seen that 
the cylindrical approximation $f(a_0,\cdots,a_m,t)$ 
indeed converges to $F([\theta],t)$ at the same 
rate at which $P_m\theta$ converges to $\theta$, 
which depends on the smoothness of $\theta\in K$.
It is worthwhile noticing that the convergence plots in 
Figures \ref{fig:functionals}-\ref{fig:FDE} are essentially 
a rescaled version of the same plot. The reason is that 
the FDE solution has continuous Fr\'echet derivatives up 
to any desired order. Hence, by the mean value Theorem 
\ref{MVT}, the convergence slopes are determined by the 
rate at which $\left\|\theta-P_m\theta\right\|_{L_p^2([0,2\pi])}$ 
goes to zero. 
\begin{figure}
\centerline{\footnotesize\hspace{0.7cm}Power law decay $|c_k|\sim 1/k^\alpha$ \hspace{4.8cm} Exponential decay $|c_k|\sim 1/\beta^k$}
\centerline{
\includegraphics[height=6.3cm]{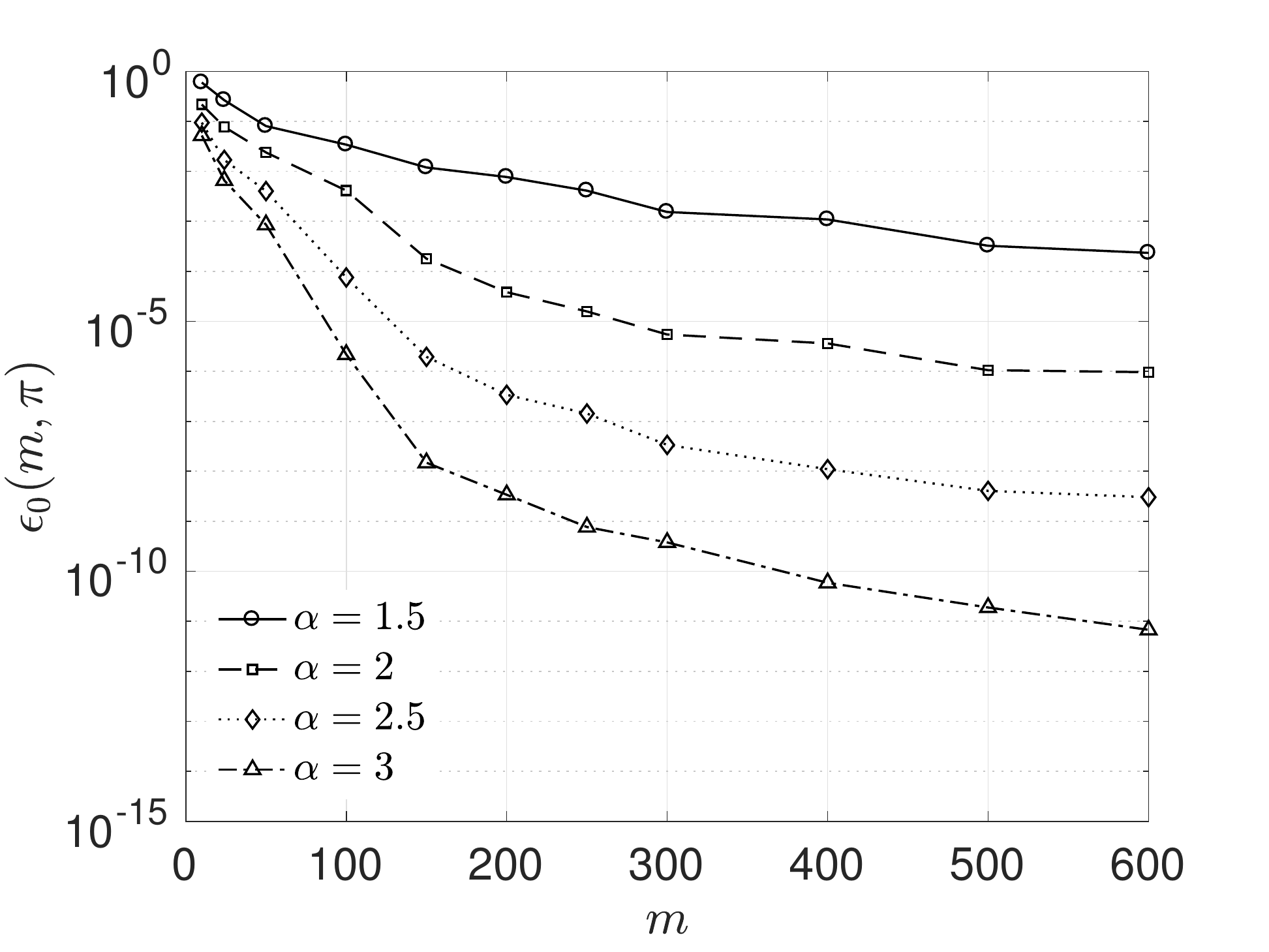}
\includegraphics[height=6.3cm]{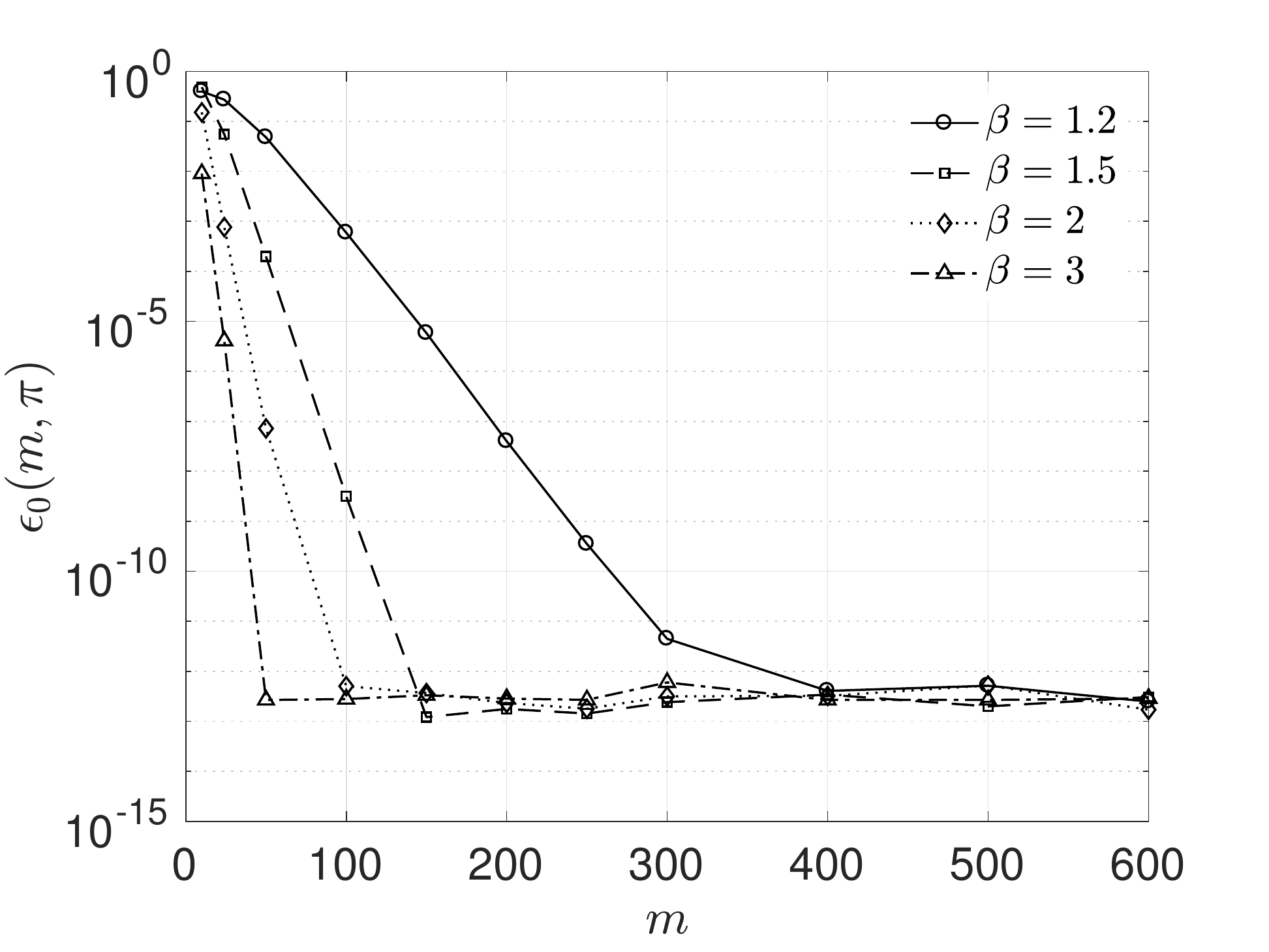}
}
\caption{Convergence of the PDE solution \eqref{PDEsol} 
to the FDE solution \eqref{FDEsol} as we increase the 
number of variables $m$. Note that the convergence 
rate of the PDE \eqref{example2C} to the FDE \eqref{example2} is, 
as before, the same as the convergence 
rate of $P_m\theta$ to $\theta$ in the $L_p^2([0,2\pi])$ 
norm.}
\label{fig:FDE}
\end{figure}

\section{Conclusions}
\label{sec:Conclusion}

We established rigorous convergence 
results for cylindrical  approximations 
of nonlinear functionals, functional derivatives, 
and functional differential equations (FDEs) 
defined on a compact subset of a 
{real Banach space $X$ admitting a basis}. 
Such approximations are 
constructed by restricting the domain 
of the functionals and the FDEs 
to the range of a finite-dimensional 
projection acting on $X$.
In this setting, we proved that 
continuous functionals and FDEs can be 
approximated  by multivariate 
functions and multidimensional 
partial differential equations (PDEs), 
respectively. The convergence rate of 
such functional approximation can be 
exponential, depending on the 
regularity of the functional (in particular 
its Fr\'echet differentiability), and its 
domain. {Rapidly converging
approximations allow us to represent 
nonlinear functionals and FDEs 
in terms of multivariate functions and PDEs
involving fewer independent variables.}
We also provided necessary and sufficient 
conditions for consistency, stability and 
convergence of functional approximations 
schemes to compute the solution of 
linear FDEs. {The main results are 
Theorem \ref{thm:FDEconvergence} and Theorem 
\ref{thm:FDEconvergence_Banach}, which are} 
based on the Trotter-Kato 
approximation theorem for abstract 
evolution equations in Banach spaces. 
The results presented in this paper 
open the possibility to utilize 
techniques for high-dimensional 
function representation such as deep neural 
networks \cite{Raissi,Raissi1,Zhu2019} 
and numerical tensor methods 
\cite{Alec2019,Bachmayr,Bram2019,
parr_tensor,approx_rates,Kolda,Alec2020,Alec2021,Bram2019} to 
approximate nonlinear functionals in 
terms of high-dimensional functions, 
and to compute approximate solutions of 
functional differential equations by 
solving high-dimensional PDEs.

\vspace{0.5cm}
\noindent 
{\bf Acknowledgements} 
This research was  supported by the U.S. Army 
Research Office (ARO) through the 
grant W911NF1810309.
Parts of this paper were completed while 
Daniele Venturi was in residence at the 
Institute for Computational and Experimental 
Research in Mathematics (ICERM) in Providence, RI, 
during the semester program ``Model and 
dimension reduction in uncertain and dynamic 
systems'', supported by the NSF-DMS grant 
1439786. 


\vspace{0.5cm}
\noindent 
{\bf Conflict of interest statement} 
On behalf of all authors, the corresponding author states that there is no conflict of interest.

\appendix

\section{Cylindrical approximation of functional integrals in real separable Hilbert spaces}
\label{sec:functionalintegral}
In this appendix we study approximation of 
functional integrals defined on a real separable 
Hilbert space $H$, with particular 
emphasis on integrals involving cylindrical functionals, 
i.e., functionals of the form \eqref{cyl}. 
This topic was first investigated by Friedrichs 
and Shapiro \cite{Friedrichs,Friedrichs_Book}, 
and it fits the framework of functional approximations
we discussed in section \ref{sec:finiteDimapprox}.
To describe the method, we first recall that $\{P_m\}$ 
is a hierarchical and complete sequence of 
orthogonal projections, 
i.e.,  $P_m\subseteq P_{m+1}$ (meaning that 
the range of $P_m$ is a subset of the range of $P_{m+1}$).  
Following Friedrichs, Shapiro and Sokorohod
\cite{Friedrichs,Friedrichs_Book,Skorohod} (see 
also \cite[Appendix B.1]{venturi2018numerical})
we define the functional integral over $H$ 
relative to a measure $\mu([\theta])$ as 
\begin{equation}
\int_H F([\theta])d\mu([\theta]) = 
\lim_{m\rightarrow \infty} \int_{D_m} F([P_m\theta])d\mu([P_m\theta]), 
\label{fInt}
\end{equation}
{where $D_m$ is defined in \eqref{Dspan}}.
We pointed out in section \ref{sec:finiteDimapprox} that 
$F([P_m\theta])=f(a_1,\cdots, a_m)$ is a $m$-dimensional 
function depending on the variables 
$a_k=(\theta,\varphi_k)_H$, which are the 
coordinates of $\theta$ relative to the orthonormal 
basis $\{\varphi_1,\varphi_2,\ldots,\varphi_m\}$.
The finite-dimensional measure in each subspace 
$D_m$ can be taken, e.g., 
as a Gaussian product measure
\begin{equation}
d\mu([P_m\theta]) = \frac{1}{(\sqrt{2\pi})^{m}}\exp\left[ -\frac{1}{2}\sum_{k=1}^m a_k^2\right]da_1\cdots da_m,
\label{functionalmeasure}
\end{equation}
which is is absolutely continuous \cite[\S 17]{Skorohod}, and 
invariant \cite[Ch. VI]{Friedrichs_Book} under 
unitary coordinate transformations\footnote{Recall that 
the coordinate system of $D_m$ is $(a_1,\ldots,a_m)$, 
and it depends on the choice of the orthonormal basis 
$\{\varphi_1,\varphi_2,\ldots,\varphi_m\}$.} 
in $D_m$. Such transformations
are induced by unitary transformations of the basis 
functions $\{\varphi_1,\varphi_2,\ldots\}$ in $H$.
There is a well-defined integration theory for
\begin{equation}
\int_{D_m} F([P_m\theta])d\mu([P_m\theta]) = 
\frac{1}{(\sqrt{2\pi})^{m}} \int_{\mathbb{R}^m} f(a_1,\ldots, a_m)
\exp\left[ -\frac{1}{2}\sum_{k=1}^m a_k^2\right]da_1\cdots da_m.
\label{fdint}
\end{equation}
At this point we recall that cylindrical functionals $F([P_m\theta])$ 
converge to $F([\theta])$ uniformly in $\theta$, 
if $\theta$ is chosen a compact subset of $H$. 
Also, cylindrical functionals are completely invariant 
in the sense of \cite[Ch V, \S III]{Friedrichs_Book}.
This guarantees that the limit in \eqref{fInt} 
exists, and that the functional integral is well-defined.
This allows us to define the inner 
product between two cylindrical 
functionals $F$ and $G$ as 
(see \cite[\S 5.1]{venturi2018numerical})
\begin{equation}
\left(F,G\right)_{\mathcal{F}}=\lim_{m\rightarrow \infty} 
\int_{D_m} F([P_m\theta])G([P_m\theta])d\mu([P_m\theta]),
\label{IP}
\end{equation}
where $\mathcal{F}$ is the vector space of functionals 
defined on the Hilbert space $H$. The inner product 
\eqref{IP} induces the norm
\begin{equation}
\left\| F([\theta])\right\|^2_{\mathcal{F}}=
\left(F,F\right)_{\mathcal{F}}.
\label{normFUNC}
\end{equation}

\vs
\noindent
{\em Example 1:} Consider the nonlinear functional  
\begin{equation}
F([\theta])=\frac{\pi}{\pi+\left(\theta,\sin(x)\right)_{L_p^2([0,2\pi])}^2}
\label{fin1}
\end{equation}
in the space of square-integrable periodic functions in 
$[0,2\pi]$, i.e., $L^2_p([0,2\pi])$. We are interested 
in computing the functional integral 
\begin{equation}
\int_{L_p^2([0,2\pi])} F([\theta])d\mu([\theta]),
\label{fin2}
\end{equation}
where the measure $d\mu([\theta])$ is the limit of the 
product measure \eqref{functionalmeasure}
as $m\rightarrow \infty$. 
To this end, we first project $\theta$ onto the orthonormal Fourier basis 
\begin{equation}
P_{2m+1} \theta = \frac{a_0}{2\pi}+
\sum_{k=1}^m a_k \frac{\sin(k x)}{\sqrt{\pi}}+
\sum_{k=1}^m b_k \frac{\cos(k x)}{\sqrt{\pi}},
\label{g4}
\end{equation}
where $a_k=\left(\theta,\sin(k x)\right)_{L_p^2([0,2\pi])}/\sqrt{\pi}$ 
and $b_k=\left(\theta,\cos(k x)\right)_{L_p^2([0,2\pi])}/\sqrt{\pi}$, 
($k=1,\ldots,m$).
A substitution of \eqref{g4} into \eqref{fin1} yields
\begin{align}
\int_{D_m} F([P_m\theta])d\mu([P_m\theta]) &=
\frac{1}{\sqrt{2\pi}} \int_{-\infty}^\infty  
 \frac{e^{-a_1^2/2}}{1+a_1^2}da_1\nonumber\\
 &= -\sqrt{\frac{e\pi}{2}}\left(\text{erf}\left(\frac{\sqrt{2}}{2}\right)-1\right),
\label{fdint00}
\end{align}
independently of $m$. Hence, the functional integral 
\eqref{fin2} is 
\begin{equation}
\int_{L_p^2([0,2\pi])} \frac{\pi d\mu([\theta])}
{\pi+\left(\theta,\sin(x)\right)_{L_p^2}^2} = 
-\sqrt{\frac{e\pi}{2}}
\left(\text{erf}\left(\frac{\sqrt{2}}{2}\right)-1\right).
\end{equation}

\section{Distance between function spaces and approximability of nonlinear functionals}
\label{sec:approximability_of_functionals}
A key concept when approximating a nonlinear functional 
$F([\theta])$ by restricting its domain $D(F)$ to a 
finite-dimensional space functions $D_m$ is the 
distance between $D_m$ and $D(F)$. 
Such distance can be quantified in different 
ways (see, e.g., \cite{Pinkus}). For example we 
can define the {\em deviation} of $D_m$ from 
$D(F)$ as
\begin{equation}
E(D_m,D(F)) = \adjustlimits \sup_{\theta\in D(F)}
\inf_{\theta_m\in D_m} \left\|\theta-\theta_m\right\|_{H}.
\label{functionaldeviation}
\end{equation}
The number $E$ measures the extent to which the worst 
element of $D(F)$ can be approximated from $D_m$. One 
may also ask how well we can approximate 
$D(F)\subseteq X$ with  $m$-dimensional subspaces 
of $H$ which are allowed to vary within $H$. 
A measure of such approximation is given by 
the Kolmogorov $m$-width 
\begin{equation}
d_m(H,D(F))= \adjustlimits \inf_{D_m\subseteq H} 
\sup_{\theta\in D(F)}
\inf_{\theta_m\in D_m} \left\|\theta-\theta_m\right\|_{H}, 
\label{knw}
\end{equation}
which quantifies the error of the {best approximation} 
to the elements of $D(F)$ by elements in a vector 
subspace $D_m\subseteq H$ of  dimension at most $m$.
The Kolmogorov $m$-width can be rigorously 
defined, e.g., for nonlinear functionals in Hilbert 
spaces (\cite{Pinkus}, Ch. 4). 
It should be emphasized that for a given domain 
of interest $D(F)$, finding the optimal basis 
spanning $D_m$ and minimizing the deviation 
$E(D,D_m)$ is not an easy task. In some cases, 
however, asymptotic results are available, e.g., 
in the case of periodic Sobolev spaces 
\cite{approx_rates}. 
It is important to emphasize that the approximation 
error and the computational complexity of 
approximating a nonlinear functional 
depends on the domain $D(F)$ and the 
choice of basis $\{\varphi_1,\varphi_2,\ldots\}$ 
spanning $D_m$. In particular, an accurate functional 
approximation may be low-dimensional in one function 
space (i.e., for small $\epsilon$, $m$ is also small) and 
high-dimensional in another (i.e., for small $\epsilon$, 
$m$ must be taken very large) -- see \S 3.1.2 in 
\cite{venturi2018numerical} for examples.

\bibliographystyle{plain}
\bibliography{main}

\end{document}